\documentclass[notitlepage,11pt,reqno]{amsart}
\usepackage[foot]{amsaddr}
\usepackage{amssymb,nicefrac,bm,upgreek,mathtools,verbatim}
\usepackage[final]{hyperref}
\usepackage[mathscr]{eucal}
\usepackage{dsfont}
\usepackage[normalem]{ulem}
\usepackage{amsopn,esint}
\usepackage[T1]{fontenc}
\usepackage{enumerate}
\usepackage[shortlabels]{enumitem}
\usepackage{bookmark}
\usepackage{wasysym}
\usepackage{esint}
\usepackage[titletoc,toc,page]{appendix}
\usepackage[utf8]{inputenc}
\usepackage{hyperref}
\usepackage{graphicx}   
\usepackage{tikz}       
\usepackage{caption}    
\usepackage{float}      
\usepackage{pgfplots}
\pgfplotsset{width=10cm,compat=1.9}
\usepgfplotslibrary{external}
\usepackage{graphicx}
\usepackage{subcaption}
\usepackage{apptools}
\AtAppendix{\counterwithin{lemma}{section}} 

\usepackage[margin=1in]{geometry}
\newcommand{\stkout}[1]{\ifmmode\text{\sout{\ensuremath{#1}}}\else\sout{#1}\fi}

\newtheorem{lemma}{Lemma}[section]
\newtheorem{theorem}{Theorem}[section]

\theoremstyle{definition}
\newtheorem{definition}{Definition}[section]

\newtheorem{remark}{Remark}[section]
\numberwithin{theorem}{section}
\numberwithin{equation}{section}

\hypersetup{
  colorlinks=true,
  citecolor=mblue,
  linkcolor=mblue,
  urlcolor = blue,
  anchorcolor = blue,
  frenchlinks=false,
  pdfborder={0 0 0},
  naturalnames=false,
  hypertexnames=false,
  breaklinks}
\usepackage[capitalize,nameinlink]{cleveref}

\usepackage[abbrev,msc-links,nobysame,citation-order]{amsrefs}

\crefname{section}{Section}{Sections}
\crefname{subsection}{Section}{Sections}
\crefname{condition}{Condition}{Conditions}
\crefname{hypothesis}{Hypothesis}{Conditions}
\crefname{assumption}{Assumption}{Assumptions}
\crefname{lemma}{Lemma}{Lemmas}
\crefname{fact}{Fact}{Facts}

\Crefname{figure}{Figure}{Figures}

\crefformat{equation}{\textup{#2(#1)#3}}
\crefrangeformat{equation}{\textup{#3(#1)#4--#5(#2)#6}}
\crefmultiformat{equation}{\textup{#2(#1)#3}}{ and \textup{#2(#1)#3}}
{, \textup{#2(#1)#3}}{, and \textup{#2(#1)#3}}
\crefrangemultiformat{equation}{\textup{#3(#1)#4--#5(#2)#6}}%
{ and \textup{#3(#1)#4--#5(#2)#6}}{, \textup{#3(#1)#4--#5(#2)#6}}%
{, and \textup{#3(#1)#4--#5(#2)#6}}

\Crefformat{equation}{#2Equation~\textup{(#1)}#3}
\Crefrangeformat{equation}{Equations~\textup{#3(#1)#4--#5(#2)#6}}
\Crefmultiformat{equation}{Equations~\textup{#2(#1)#3}}{ and \textup{#2(#1)#3}}
{, \textup{#2(#1)#3}}{, and \textup{#2(#1)#3}}
\Crefrangemultiformat{equation}{Equations~\textup{#3(#1)#4--#5(#2)#6}}%
{ and \textup{#3(#1)#4--#5(#2)#6}}{, \textup{#3(#1)#4--#5(#2)#6}}%
{, and \textup{#3(#1)#4--#5(#2)#6}}

\crefdefaultlabelformat{#2\textup{#1}#3}
\newcommand{\vertiii}[1]{{\left\vert\kern-0.25ex\left\vert\kern-0.25ex\left\vert #1
    \right\vert\kern-0.25ex\right\vert\kern-0.25ex\right\vert}}

\usepackage{color}
\definecolor{dmagenta}{rgb}{.4,.1,.4}
\definecolor{dblue}{rgb}{.0,.0,.5}
\definecolor{mblue}{rgb}{.0,.0,.7}
\definecolor{ddblue}{rgb}{.0,.0,.4}
\definecolor{dred}{rgb}{.9,.0,.0}
\definecolor{dgreen}{rgb}{.0,.5,.0}
\definecolor{Eeom}{rgb}{.0,.0,.5}
\definecolor{dbrown}{rgb}{.6,.0,.0}

\allowdisplaybreaks

\makeatletter
\renewcommand\subsection{\@startsection{subsection}{2}%
  \z@{-3.25ex\@plus-1ex \@minus-.2ex}%
  {1.5ex\@plus.2ex}%
  {\normalfont\normalsize\bfseries}}
\makeatother
\makeatletter
\renewcommand\section{\@startsection{section}{1}%
  \z@{-3.5ex \@plus -1ex \@minus -.2ex}%
  {2.3ex \@plus .2ex}%
  {\normalfont\large\bfseries}}
\makeatother
\newcommand{\ttl}{\Large Existence, uniqueness and Large time behavior of a viscous conservation law with discontinuous flux} 
\begin{document}
\title[Large time behavior]
{\ttl}

\author{Smritikana Pal$^1$}
\author{Manas Ranjan Sahoo$^2$}

\address{$^1$ $^2$ School of Mathematical Sciences, National Institute of Science Education and Research,  Bhubaneswar, P.O. Jatni, Khurda, Odisha 752050, India. An OCC of Homi Bhabha National Institute, Anushaktinagar, Mumbai, Maharashtra, India - 400 094. Email: smritikana.pal@niser.ac.in, manas@niser.ac.in}



\begin{abstract}
In this paper, we study a viscous conservation law with discontinuous flux. We define a weak solution concept, show its existence via an explicit formula, and prove that the weak solution is unique. In addition, we obtain the large-time behavior of the solution. The asymptotic limit of the solution is a steady-state solution. Here, the large-time asymptotic governed by the convection terms either converges to a constant steady state or to a non-constant steady state, depending on the behavior of the flux functions at a threshold value, which is the meeting point of two fluxes. This phenomenon does not occur in the single-flux case.
\end{abstract}

\maketitle
\maketitle
\tableofcontents

\section{Introduction and Main results}
The pioneering work of Hopf \cite{hopf} addressed several important questions for the Burgers equation. He obtained the large-time behavior of the viscous Burgers equation and the inviscid Burgers equation. Later, this important paper opened many research frontiers. Basically, the study of large-time behavior was done by replacing the Burgers-type flux by a general flux, see \cites{Dix2002,Cole1951, HattoriNishihara1985, GallayWayne2002, KimTzavaras2001, Dalibard, Needham}. Some of the recent works in this direction, we would like to cite \cite{GhoulMasmoudiPacherie2025} where it is proved that for Burgers-type equations, the nonlinear transport term creates stronger mixing effect. E. Pacherie \cite{Pacherie2024}, develops a hierarchy of asymptotic profiles in the infinite mass setting. Huang–Qiu–Wang–Yang\cite{Huang} studied the nonlinear asymptotic stability of a non-self-similar rarefaction wave for the two-dimensional viscous Burgers equation. There is ample literature on the study of conservation laws with discontinuous flux without viscosity; see \cite{MR2028700, MR2323047} and references therein. The large-time asymptotic for viscous conservation laws with discontinuous fluxes is not well studied. Viscous conservation laws with discontinuous flux read:
\begin{equation}\label{IVP for CDF}
   \begin{split}
      u_t+\Big[(1-H(x)) f(u)+H(x)g(u)\Big]_x-u_{xx}&=0,~~~~~~~x\in\mathbb{R}, t>0\\
       u(x,0)&=u_0(x),~~x\in\mathbb{R}
   \end{split}
\end{equation} 
where $H$ is the Heaviside function and $f, g:\mathbb{R} \to \mathbb{R}$ are $C^2$ functions. 

The above equation describes a simplified model of two-phase flow in heterogeneous porous media with capillary diffusion, and studying its long-time behavior reveals how the saturation evolves toward equilibrium. This equation can also be used to understand transport through heterogeneous porous media. The flux discontinuity reflects abrupt changes in material properties, such as permeability or porosity, across interfaces separating different rock types. The viscous term accounts for physical diffusion due to capillary effects, molecular diffusion, or mechanical dispersion. Understanding the large-time behavior of solutions is important for predicting the long-term evolution of transport processes and the eventual equilibrium state of the system. In this direction, we would like to cite \cite{AndreianovCances2013, Kaasschieter1999, SpaydShearer2011} for the Buckley-Leverett equation. 

The study of the large-time behavior of the solution of the equation \eqref{IVP for CDF} is the general aim. This problem presents several challenges. In particular, it is necessary to formulate an appropriate notion of a solution, establish its existence and uniqueness, and investigate its large-time behavior. Chung et al. \cite{chung} studied the existence, uniqueness, and large time behavior for the Burgers equation with stationary point source, namely
\begin{equation*}
    u_t+uu_x=u_{xx}+\delta.
\end{equation*}

This equation can be considered as a particular form of \eqref{IVP for CDF} for $f(u)=\frac{u^2}{2}$ and $g(u)=\frac{u^2}{2}-1.$  Note that in this case, the two fluxes $f(u)=\frac{u^2}{2}$ and $g(u)=\frac{u^2}{2}-1$ do not agree at any point. If we consider two fluxes that agree at a single point, the problem is different and will exhibit different asymptotic behavior depending on the meeting point. The aim of this paper is to consider such a case. As a starting point, we intend to establish the existence, uniqueness, and large-time asymptotics of the following equation (with $f(u)=\frac{u^2}{2}$ and $g(u)=\frac{(u+\alpha)^2}{2}$) where $f(u)$ and $g(u)$ agree at $u=-\frac{\alpha} {2}.$ 

\begin{equation}\label{IVP}
   \begin{split}
      u_t+\Big[(1-H(x))\frac{u^2}{2}+H(x)\frac{(u+\alpha)^2}{2}\Big]_x-u_{xx}&=0,~~~~~~~x\in\mathbb{R}, t>0\\
       u(x,0)&=u_0(x),~~x\in\mathbb{R}
   \end{split}
\end{equation} 
where $H$ is the Heaviside function, $\alpha \neq 0$ and $u_0 \in L^1_{loc}(\mathbb{R}).$

To obtain the weak solution, we first study two initial-boundary value problems posed on the first and second quadrants, respectively. Using the solutions of these auxiliary problems, we construct an intermediate function $G(t)$ that connects these two problems. Taking the Hopf-Cole type transformation, we obtain an explicit formula described below.

\begin{theorem}
\label{explicit solution}
  The explicit solution of equation \eqref{IVP} is given by   
  \begin{equation}\label{explicit formula}
    \begin{aligned}
        u(x,t)=\begin{cases}
        -2\frac{\phi_x(x,t)}{\phi(x,t)},&x<0, ~~t>0\\
   -2\frac{\psi_x(x,t)}{\psi(x,t)}-\alpha, &x>0, ~~ t>0
    \end{cases}
    \end{aligned} 
    \end{equation}
    where $\phi,\psi,\phi_x$ and $\psi_x$ are as follows:
    \begin{equation*}
        \begin{aligned}
            &\phi(x,t)=\frac{-1}{2\sqrt{\pi t}}\int_0^\infty \Big\{e^{-\frac{(x-\xi)^2}{4t}}-e^{-\frac{(x+\xi)^2}{4t}}\Big\}\phi_0(-\xi)d\xi-\frac{x}{2\sqrt{\pi}}\int_0^te^{-\frac{x^2}{4 (t-\tau)}}\frac{G(\tau)}{(t-\tau)^{\frac{3}{2}}}d\tau\\
            & \psi(x,t)=\frac{1}{2\sqrt{\pi t}}\int_0^\infty \Big\{e^{-\frac{(x-\xi)^2}{4t}}-e^{-\frac{(x+\xi)^2}{4t}}\Big\}\phi_0(\xi)d\xi+\frac{x}{2\sqrt{\pi}}\int_0^te^{-\frac{x^2}{4(t-\tau)}}\frac{G(\tau)}{(t-\tau)^{\frac{3}{2}}}d\tau\\
            & \phi_x(x,t)  =\frac{-1}{4\sqrt{\pi}(t)^{\frac{3}{2}}}\int_0^{\infty}\phi_0(-\xi)\Big\{(\xi-x)e^{-\frac{(x-\xi)^2}{4t}}+(x+\xi)e^{-\frac{(x+\xi)^2}{4t}}\Big\}d\xi\\&+\int_0^t\frac{G^{\prime}(\tau)}{\sqrt{\pi(t-\tau)}}e^{\frac{-x^2}{4(t-\tau)} }d\tau+\frac{G(0)}{\sqrt{\pi t}}e^{\frac{-x^2}{4t}}\\
            &  \psi_x(x,t) =\frac{1}{4\sqrt{\pi}(t)^{\frac{3}{2}}}\int_0^{\infty}\phi_0(\xi)\Big\{(\xi-x)e^{-\frac{(x-\xi)^2}{4t}}+(x+\xi)e^{-\frac{(x+\xi)^2}{4t}}\Big\}d\xi\\&-\int_0^t\frac{G^{\prime}(\tau)}{\sqrt{\pi(t-\tau)}}e^{\frac{-x^2}{4(t-\tau)} }d\tau-\frac{G(0)}{\sqrt{\pi t}}e^{\frac{-x^2}{4t}}.
        \end{aligned}
    \end{equation*}
    Here 
    \begin{equation*}
        \begin{aligned}
            G(t)&=\frac{1}{\sqrt{\pi}}\int_0^t\int_0^{\infty}\frac{\xi e^{\frac{-\xi^2}{4s}}\Big(\phi_0(-\xi)+\phi_0(\xi)\Big)}{4(s)^{\frac{3}{2}}\sqrt{t-s}}\Big[\frac{1}{\sqrt{\pi}}+\frac{\alpha\sqrt{t-s}}{4}e^{\frac{\alpha^2(t-s)}{16}}\operatorname{erfc}\Big(\frac{-\alpha\sqrt{t-s}}{4}\Big)\Big]d\xi ds,
        \end{aligned}
    \end{equation*}
   and $$\phi_0(x)=\begin{cases}
    exp\Big(-\frac{1}{2}\int_0^{x}u_0(y)dy\Big)&x<0\\
    exp\Big(-\frac{1}{2}\int_0^{x}\Big[u_0(y)+\alpha\Big] dy\Big)&x>0.
\end{cases}$$
The explicit formula satisfies the weak formulation \eqref{integral identity}.
\end{theorem}
Next, we prove the uniqueness; see theorem \eqref{uniqueness theorem}. To investigate the large-time behavior of the weak solution, we first analyze the asymptotic behavior of an intermediate function $G$ together with that of the solutions to the associated initial-boundary value problems. By combining these results with the explicit representation formula, we obtain the large-time asymptotic behavior of the weak solution.

{\bf For large-time behavior, $\alpha<0$ and $\alpha>0$ play important roles. The left and right fluxes meet each other at $u=-\frac{\alpha}{2}$. At this point, for $\alpha>0,$ the speed $(g^{\prime}(u)=u+\alpha)$ of the right flux is positive, and the speed $(f^{\prime}(u)=u)$ due to the left flux is negative. Herustically, at this point, the data are outgoing, and the smoothing effect $u_{xx}$ forces the solution to converge to a constant steady state. Since, for $\alpha < 0$, the data are incoming, this prevents the system from converging to a constant steady state. In this case, we obtain a non-constant steady state. The large-time asymptotic is continuous due to the smoothing effect.} 

The above is rigorously justified in the following two theorems under the assumption 
\begin{equation}
\label{extracondition}
\begin{aligned}
&\theta \in C^2(\mathbb{R}) \cap W^{2,\infty}(\mathbb{R})\\
  & u_0 \in L^1((-\infty,0)), \quad u_0+\alpha \in L^1((0,\infty))
  \end{aligned}
\end{equation} where $\theta(\xi):=\Big(\phi_0(-\xi)+\phi_0(\xi)\Big)$ for $\xi>0.$ 

For $\alpha<0$, as discussed above, we have the following theorem, where the solution approaches a non-constant steady state as $t\to \infty$. 
\begin{theorem}
\label{limit of solution}
    Let $u$ be the weak solution of \eqref{IVP} with an initial value $u_0.$ If condition \eqref{extracondition} is satisfied along with $\alpha<0$, then as $t\to \infty$ \begin{equation*}
        u(x,t)\to\begin{cases}
       \frac{-2\alpha\phi_0(-\infty) }{2\Big[\phi_0(-\infty)+\phi_0(\infty)\Big]+\alpha \phi_0(-\infty)x}, &x<0\\
        \frac{\alpha\Big[2\phi_0(-\infty)-\alpha \phi_0(\infty)x\Big]}{\alpha \phi_0(\infty)x-2\Big[\phi_0(-\infty)+\phi_0(\infty)\Big]}, &x>0.
    \end{cases}
    \end{equation*}
\end{theorem}

For $\alpha>0$, as discussed above, we have the following theorem, where the solution approaches a constant steady state as $t\to \infty$.
\begin{theorem}
\label{limit of solution 1}
    Let $u$ be the weak solution of \eqref{IVP} with an initial value $u_0.$ If   conditions \eqref{extracondition}  are satisfied along with $\alpha>0,$ then as $t\to \infty$ \begin{equation*}
            u(x,t)\to -\frac{\alpha}{2}, ~~~~~~~x\in \mathbb{R}.
    \end{equation*}
\end{theorem}

\begin{remark}
Consider a more general problem as follows:
    \begin{equation*}
   \begin{split}
      u_t+\Big[(1-H(x))\frac{(u+\alpha)^2}{2}+H(x)\frac{(u+\beta)^2}{2}\Big]_x-u_{xx}&=0,~~~~~~~x\in\mathbb{R}, t>0\\
       u(x,0)&=u_0(x),~~x\in\mathbb{R}
   \end{split}
\end{equation*} 
 where $H$ is the Heaviside function, $u_0$ is a measurable function and $\alpha,\beta  \neq 0.$ We introduce the change of variables $v=u+\alpha.$ Under this transformation, the problem reduces to an equivalent form of \eqref{IVP}. Consequently, it is enough to study \eqref{IVP}.
 \end{remark}

The remainder of the paper is structured as follows. In Section 2, we establish the existence and uniqueness of weak solutions to \eqref{IVP} and derive an explicit representation formula. In Sections 3 and 4, we investigate the large-time asymptotic behavior of an intermediate function $G$, which plays a crucial role in determining the asymptotic behavior of the weak solution. We show that the weak solution converges to a continuous steady-state solution. In addition, for a particular choice of initial data, we present numerical illustrations demonstrating the convergence of the explicit solution to the steady-state solution in the cases $\alpha<0$ and $\alpha>0$, respectively.
\section{Existence and uniqueness of solution}
First, we introduce a suitable concept of weak solution for the problem \eqref{IVP}, namely,
\begin{equation*}
   \begin{split}
      u_t+\Big[(1-H(x))\frac{u^2}{2}+H(x)\frac{(u+\alpha)^2}{2}\Big]_x-u_{xx}&=0,~~~~~~~x\in\mathbb{R}, t>0\\
       u(x,0)&=u_0(x),~~x\in\mathbb{R}.
   \end{split}
\end{equation*} 
\begin{definition}
A function $u\in C([0,\infty); H^1(\mathbb{R})\cap L^{\infty}(\mathbb{R}))$ is said to be a weak solution of equation \eqref{IVP} if the following integral identity holds
\begin{equation}
\begin{split}
     \label{integral identity}
    \int_{0}^{\infty} \int_{-\infty}^{\infty}\Big[ u\zeta_t + [(1-H(x))\frac{u^2}{2}+H(x)\frac{(u+\alpha)^2}{2}-u_x] \zeta_x\Big] dx dt \\+\int_{-\infty}^{\infty} u_0(x)\zeta(x,0) dx=0
    \end{split}
\end{equation}
for any test function $\zeta \in C_c^{\infty}(\mathbb{R}\times [0, \infty)).$
\end{definition}
\subsection{Existence of weak solution}

To construct a weak solution of \eqref{IVP}, we proceed in several steps, each formulated as a lemma. In the first lemma, we start with two boundary value problems for the heat equation.

\begin{lemma}
\label{solution for phi and psi}
Let $G:[0, \infty) \to \mathbb{R}$ be a $C^1$ function and $$\phi_0(x)=\begin{cases}
    exp\Big(-\frac{1}{2}\int_0^{x}u_0(y)dy\Big)&x<0\\
    exp\Big(-\frac{1}{2}\int_0^{x}\Big[u_0(y)+\alpha\Big] dy\Big)&x>0.
\end{cases}$$ Then the solutions of the initial and boundary value problems 
    \begin{equation*}
   \begin{cases}
        \phi_t-\phi_{xx}=0,~~ x<0,~ t>0\\
        \phi(x,0)=\phi_0(x),~~ x<0\\
        \phi(0,t)=G(t),~ t>0
   \end{cases}
\end{equation*} 

and 

\begin{equation*}
   \begin{cases}
        \psi_t-\psi_{xx}=0,~~x>0, t>0\\
        \psi(x,0)=\phi_0(x),~~x>0\\
        \psi(0,t)=G(t),~~t>0
   \end{cases}
\end{equation*}

are given by 
\begin{equation}
\label{formula for phi}
\phi(x,t)=\frac{-1}{2\sqrt{\pi t}}\int_0^\infty \Big\{e^{-\frac{(x-\xi)^2}{4t}}-e^{-\frac{(x+\xi)^2}{4t}}\Big\}\phi_0(-\xi)d\xi-\frac{x}{2\sqrt{\pi}}\int_0^te^{-\frac{x^2}{4 (t-\tau)}}\frac{G(\tau)}{(t-\tau)^{\frac{3}{2}}}d\tau
\end{equation} and 
\begin{equation}\label{formula for psi}
        \psi(x,t)=\frac{1}{2\sqrt{\pi t}}\int_0^\infty \Big\{e^{-\frac{(x-\xi)^2}{4t}}-e^{-\frac{(x+\xi)^2}{4t}}\Big\}\phi_0(\xi)d\xi+\frac{x}{2\sqrt{\pi}}\int_0^te^{-\frac{x^2}{4(t-\tau)}}\frac{G(\tau)}{(t-\tau)^{\frac{3}{2}}}d\tau.
\end{equation}
If $G(t)>0$ for $t>0,$ then $\phi(x,t)~~~, \psi(x,t)>0$ for $t>0.$
\end{lemma}
\begin{proof}
Applying the results [see \cite{Kevorkian} p. 18 and p. 22], we obtain \eqref{formula for phi} 
and \eqref{formula for psi}. Since $G(t)>0$ for $t>0$ and $\phi(x,0), \psi(x,0)>0,$ from the 
explicit formulae,  it is easy to check that $\phi(x,t)>0$ for $x<0, t>0$ and $\psi(x,t)>0$ for 
$x>0,\,\, t>0$.
\end{proof}
In the following lemma, we specify the unknown function $G(t)$ using a regularity condition. Later, we show that this regularity condition is required to obtain the solution of the problem \eqref{IVP} in the sense of the definition\eqref{integral identity}.
\begin{lemma}
\label{properties of G}
 Define $\theta(\xi):=\Big(\phi_0(-\xi)+\phi_0(\xi)\Big)$ for $\xi>0.$  If $\theta \in C^2(\mathbb{R})\cap W^{2,\infty}(\mathbb{R}),$ then there exists a unique positive $C^1$ function $G(t)$  such that $G(t)=\frac{2}{\alpha}\Big[\phi_x(0+,t)-\psi_x(0-,t)\Big].$ 
 The explicit formula for $G(t)$ is as follows:
 \begin{equation}
\label{formula for G}
G(t)=\frac{1}{\sqrt{\pi}}\int_0^t\int_0^{\infty}\frac{\xi e^{\frac{-\xi^2}{4s}}\Big(\phi_0(-\xi)+\phi_0(\xi)\Big)}{4(s)^{\frac{3}{2}}\sqrt{t-s}}\Big[\frac{1}{\sqrt{\pi}}+\frac{\alpha\sqrt{t-s}}{4}e^{\frac{\alpha^2(t-s)}{16}}\operatorname{erfc}\Big(\frac{-\alpha\sqrt{t-s}}{4}\Big)\Big]d\xi ds.
\end{equation}
\end{lemma}

\begin{proof}
{ \bf Step 1:}The formulae \eqref{formula for phi}and \eqref{formula for psi} can be written as \begin{equation}\label{left explicit formula}
\begin{aligned}
    \phi(x,t)=&\frac{-1}{2\sqrt{\pi t}}\int_0^\infty \Big\{e^{-\frac{(x-\xi)^2}{4t}}-e^{-\frac{(x+\xi)^2}{4t}}\Big\}\phi_0(-\xi)d\xi+ \int_0^tG^{\prime}(\tau)\operatorname{erfc}\Big(\frac{-x}{2\sqrt{t-\tau}}\Big)d\tau\\&+G(0)\operatorname{erfc}\Big(\frac{-x}{2\sqrt{t}}\Big)
\end{aligned}
\end{equation} and  \begin{equation}\label{right explicit formula}
\begin{aligned}
    \psi(x,t)=&\frac{1}{2\sqrt{\pi t}}\int_0^\infty \Big\{e^{-\frac{(x-\xi)^2}{4t}}-e^{-\frac{(x+\xi)^2}{4t}}\Big\}\phi_0(\xi)d\xi+\int_0^tG^{\prime}(\tau) \operatorname{erfc}\Big(\frac{x}{2\sqrt{t-\tau}}\Big)d\tau\\&+G(0)\operatorname{erfc}\Big(\frac{x}{2\sqrt{t}}\Big).
\end{aligned}
\end{equation}

{\bf Step 2:}
From \eqref{left explicit formula} and \eqref{right explicit formula} we get the following.
 \begin{equation}\label{definition 6}
    \begin{split}
        \phi_x(x,t)&
        =\frac{-1}{4\sqrt{\pi}(t)^{\frac{3}{2}}}\int_0^{\infty}\phi_0(-\xi)\Big\{(\xi-x)e^{-\frac{(x-\xi)^2}{4t}}+(x+\xi)e^{-\frac{(x+\xi)^2}{4t}}\Big\}d\xi\\&+\int_0^t\frac{G^{\prime}(\tau)}{\sqrt{\pi(t-\tau)}}e^{\frac{-x^2}{4(t-\tau)} }d\tau+\frac{G(0)}{\sqrt{\pi t}}e^{\frac{-x^2}{4t}}
    \end{split}
\end{equation}
and,
\begin{equation}\label{definition 5}
    \begin{split}
        \psi_x(x,t)&
        =\frac{1}{4\sqrt{\pi}(t)^{\frac{3}{2}}}\int_0^{\infty}\phi_0(\xi)\Big\{(\xi-x)e^{-\frac{(x-\xi)^2}{4t}}+(x+\xi)e^{-\frac{(x+\xi)^2}{4t}}\Big\}d\xi\\&-\int_0^t\frac{G^{\prime}(\tau)}{\sqrt{\pi(t-\tau)}}e^{\frac{-x^2}{4(t-\tau)} }d\tau-\frac{G(0)}{\sqrt{\pi t}}e^{\frac{-x^2}{4t}}.
    \end{split}
\end{equation}
Letting $x\to 0+$ in \eqref{definition 6} and $x\to 0-$ in \eqref{definition 5}, we get
\begin{equation*}
    \begin{split}
        \phi_x(0+,t)&=\frac{-1}{2\sqrt{\pi}(t)^{\frac{3}{2}}}\int_0^{\infty}\xi \phi_0(-\xi)e^{\frac{-\xi^2}{4t}}d\xi+\int_0^tG^{\prime}(\tau)\frac{1}{\sqrt{\pi(t-\tau)}}d\tau+\frac{G(0)}{\sqrt{\pi t}}
    \end{split}
\end{equation*}
and 
\begin{equation*}
\begin{split}
      \psi_x(0-,t)&=\frac{1}{2\sqrt{\pi}(t)^{\frac{3}{2}}}\int_0^{\infty}\xi\phi_0(\xi) e^{\frac{-\xi^2}{4t}}d\xi-\int_0^tG^{\prime}(\tau)\frac{1}{\sqrt{\pi(t-\tau)}}d\tau-\frac{G(0)}{\sqrt{\pi t}}.
\end{split}
\end{equation*}

Equating $G(t)=\frac{2}{\alpha}\Big[\phi_x(0+,t)-\psi_x(0-,t)\Big],$ we get
\begin{equation}\label{equality 1}
   \begin{aligned}
         &G(t)=
         \frac{2}{\alpha}\Big[\frac{-1}{2\sqrt{\pi}(t)^{\frac{3}{2}}}\int_0^{\infty}\xi e^{\frac{-\xi^2}{4t}}\Big(\phi_0(-\xi)+\phi_0(\xi)\Big)d\xi+2\int_0^t\frac{G^{\prime}(\tau)}{\sqrt{\pi(t-\tau)}}d\tau+2\frac{G(0)}{\sqrt{\pi t}}\Big].
   \end{aligned}
\end{equation}
So in the Caputo derivative sense, \eqref{equality 1} can be written as follows \begin{equation}\label{caputo DE}
    \begin{aligned}
        {}^CD_t^{\frac{1}{2}}G(t)=\frac{\alpha}{4}G(t)+\frac{f(t)}{4}
    \end{aligned}
\end{equation} where $f(t)=\frac{1}{\sqrt{\pi}t^{\frac{3}{2}}}\int_0^{\infty}\xi e^{\frac{-\xi^2}{4t}}\Big(\phi_0(-\xi)+\phi_0(\xi)\Big)d\xi-4\frac{G(0)}{\sqrt{\pi t}}.$

\vspace{0.9pt}

\noindent
Therefore, the explicit formula for \eqref{caputo DE} is given by 
\begin{equation}\label{equlity 2}
    \begin{aligned}
       G(t)&=G(0)E_{\frac{1}{2}}\Big(\frac{\alpha\sqrt{t}}{4}\Big)+\int_0^t\frac{f(s)}{4\sqrt{t-s}}E_{\frac{1}{2},\frac{1}{2}}\Big(\frac{\alpha\sqrt{t-s}}{4}\Big)ds\\&=G(0)\Big[e^{\frac{\alpha^2t}{16}}\operatorname{erfc}\Big(\frac{-\alpha\sqrt{t}}{4}\Big)\Big]+\int_0^t\frac{f(s)}{4\sqrt{t-s}}\Big[\frac{1}{\sqrt{\pi}}+\frac{\alpha\sqrt{t-s}}{4}e^{\frac{\alpha^2(t-s)}{16}}\operatorname{erfc}\Big(\frac{-\alpha\sqrt{t-s}}{4}\Big)\Big]ds.
    \end{aligned}
\end{equation} The last equality follows from \begin{equation*}
    \begin{aligned}
        &E_{\frac{1}{2}}(z)=e^{z^2}\operatorname{erfc}(-z)\\&
        E_{\frac{1}{2},\frac{1}{2}}(z)=\frac{1}{\sqrt{\pi}}+ze^{z^2}\operatorname{erfc}(-z)
        \end{aligned}
\end{equation*} where $E_{\frac{1}{2}}$ and $ E_{\frac{1}{2},\frac{1}{2}}$ are Mittag-Leffler functions.
Putting the value of $f$ in \eqref{equlity 2} and rearranging the terms, we get the following.

\begin{equation}\label{equality 3}
   \begin{aligned}
   &G(t)=\\&G(0)\Bigg[e^{\frac{\alpha^2t}{16}}\operatorname{erfc}\Big(\frac{-\alpha\sqrt{t}}{4}\Big)-\frac{1}{\sqrt{\pi}}\int_0^t\frac{1}{\sqrt{s(t-s)}}\Big[\frac{1}{\sqrt{\pi}}+\frac{\alpha\sqrt{t-s}}{4}e^{\frac{\alpha^2(t-s)}{16}}\operatorname{erfc}\Big(\frac{-\alpha\sqrt{t-s}}{4}\Big)\Big]ds\Bigg]\\&+\frac{1}{\sqrt{\pi}}\int_0^t\int_0^{\infty}\frac{\xi e^{\frac{-\xi^2}{4s}}\Big(\phi_0(-\xi)+\phi_0(\xi)\Big)}{4(s)^{\frac{3}{2}}\sqrt{t-s}}\Big[\frac{1}{\sqrt{\pi}}+\frac{\alpha\sqrt{t-s}}{4}e^{\frac{\alpha^2(t-s)}{16}}\operatorname{erfc}\Big(\frac{-\alpha\sqrt{t-s}}{4}\Big)\Big]d\xi ds.
   \end{aligned}
\end{equation}
We now show that the coefficient of $G(0)$ is $0$. To this end, define $h(t):=e^{\frac{\alpha^2t}{16}}\operatorname{erfc}\Big(\frac{-\alpha\sqrt{t}}{4}\Big).$ Observe that
\begin{equation*}
    \begin{aligned}
        \frac{d}{ds}(h(t-s))&=\frac{d}{ds}\Big[e^{\frac{\alpha^2(t-s)}{16}}\operatorname{erfc}\Big(\frac{-\alpha\sqrt{t-s}}{4}\Big)\Big]\\&=-\frac{\alpha}{4\sqrt{\pi (t-s)}}-\frac{\alpha^2}{16}e^{\frac{\alpha^2(t-s)}{16}}\operatorname{erfc}\Big(\frac{-\alpha\sqrt{t-s}}{4}\Big)\\&=-\frac{\alpha}{4\sqrt{t-s}}\Big[\frac{1}{\sqrt{\pi}}+\frac{\alpha\sqrt{t-s}}{4}e^{\frac{\alpha^2(t-s)}{16}}\operatorname{erfc}\Big(\frac{-\alpha\sqrt{t-s}}{4}\Big)\Big].
    \end{aligned}
\end{equation*} 
Then, 
\begin{equation}\label{equality 4}
    \begin{aligned}
     \text{Coefficient of $G(0)$}
        &=h(t)+\frac{4}{\alpha\sqrt{\pi}}\int_0^t\frac{1}{\sqrt{s}}\frac{d}{ds}(h(t-s))ds\\&=h(t)-\frac{4}{\alpha\sqrt{\pi}}\int_0^t\frac{1}{\sqrt{s}}h^{\prime}(t-s)ds\\&=h(t)-\frac{4}{\alpha\sqrt{\pi}}\int_0^t\frac{1}{\sqrt{t-u}}h^{\prime}(u)du\\&=h(t)-\frac{4}{\alpha}~{}^CD_t^{\frac{1}{2}}h(t).
    \end{aligned}
\end{equation}

Now we show that $h(t)=\frac{4}{\alpha}~{}^CD_t^{\frac{1}{2}}h(t).$ First observe that \begin{equation}\label{derivative of h}
    \begin{aligned}
        h^{\prime}(t)&=\frac{\alpha}{4\sqrt{\pi t}}+\frac{\alpha^2}{16}e^{\frac{\alpha^2t}{16}}\operatorname{erfc}\Big(\frac{-\alpha\sqrt{t}}{4}\Big)\\&=\frac{\alpha}{4\sqrt{\pi t}}+\frac{\alpha^2}{16}h(t).
        \end{aligned}
\end{equation} Then, by definition, we get \begin{equation}\label{caputo derivative}
    \begin{aligned}
        {}^CD_t^{\frac{1}{2}}h(t)&=\frac{1}{\sqrt{\pi}}\int_0^t\frac{1}{\sqrt{t-u}}h^{\prime}(u)du\\&=\frac{1}{\sqrt{\pi}}\int_0^t\frac{1}{\sqrt{t-u}}\Big[\frac{\alpha}{4\sqrt{\pi u}}+\frac{\alpha^2}{16}h(u)\Big]du\\&=\frac{\alpha^2}{16\sqrt{\pi}}\int_0^t\frac{h(u)}{\sqrt{t-u}}du+\frac{\alpha}{4\pi}\int_0^t\frac{1}{\sqrt{u(t-u)}}du.
    \end{aligned}
\end{equation} Now to calculate the first term in \eqref{caputo derivative} we use the identity $$\frac{1}{\sqrt{\pi}}\int_0^t\frac{f(u)}{\sqrt{t-u}}du=\frac{1}{a}(f(t)-1)~~~ \text{if}~~ f=e^{a^2t}\operatorname{erfc}(-a\sqrt{t}).$$ So we get 
\begin{equation}
\label{calculation}
    \frac{\alpha^2}{16\sqrt{\pi}}\int_0^t\frac{h(u)}{\sqrt{t-u}}du=\frac{\alpha}{4}(h(t)-1).
\end{equation} 

Therefore, from \eqref{caputo derivative} and the above calculation we get $ {}^CD_t^{\frac{1}{2}}h(t)=\frac{\alpha}{4}h(t).$ Hence, from \eqref{equality 4} $\text{Coefficient of $G(0)$}=0.$
So from \eqref{equality 3}, we get $$G(t)=\frac{1}{\sqrt{\pi}}\int_0^t\int_0^{\infty}\frac{\xi e^{\frac{-\xi^2}{4s}}\Big(\phi_0(-\xi)+\phi_0(\xi)\Big)}{4(s)^{\frac{3}{2}}\sqrt{t-s}}\Big[\frac{1}{\sqrt{\pi}}+\frac{\alpha\sqrt{t-s}}{4}e^{\frac{\alpha^2(t-s)}{16}}\operatorname{erfc}\Big(\frac{-\alpha\sqrt{t-s}}{4}\Big)\Big]d\xi ds.$$ 

\noindent
 {\bf Step 3:} In this step, we show that $G(t)>0$ for $t>0$. 
It is enough to show that
$$\Big[\frac{1}{\sqrt{\pi}}+\frac{\alpha\sqrt{t-s}}{4}e^{\frac{\alpha^2(t-s)}{16}}\operatorname{erfc}\Big(\frac{-\alpha\sqrt{t-s}}{4}\Big)\Big]>0$$ 
for $t>0.$  This follows from the following trailing calculations:
Observe that for $y>0$ 

 \begin{equation*}
    \begin{aligned}
        &\frac{2}{\sqrt{\pi}}\int_y^{\infty}e^{-z^2}dz <\frac{2}{\sqrt{\pi}}\int_y^{\infty}\frac{z}{y}e^{-z^2}dz=-\frac{1}{\sqrt{\pi}y}\int_y^{\infty}\frac{d}{dz}e^{-z^2}dz=\frac{1}{\sqrt{\pi}y}e^{-y^2}.
    \end{aligned}
\end{equation*} 
This implies \[y e^{y^2} \operatorname{erfc}(y)<\frac{1}{\sqrt{\pi}}.\]
For $\alpha < 0$, choosing
\[
y = \frac{-\alpha \sqrt{t - s}}{4},
\]
the claim in \textbf{Step 3} follows. For $\alpha>0,$ $G(t)$ is strictly positive, which is obvious. \\

\noindent
{\bf Step 4:} In this step, we prove that $G$ is a $C^1$ function on $(0, \infty).$
To this end, we can rewrite 
\begin{equation}
\label{convolution}
    G(t)=\int_0 ^{t} F(s) H(t-s) ds=F*H(t),
\end{equation}
where 
\begin{equation*}
\begin{aligned}
F(t) &= \frac{1}{\sqrt{\pi}} \int_0^{\infty}\frac{\xi e^{\frac{-\xi^2}{4t}}\Big(\phi_0(-\xi)+\phi_0(\xi)\Big)}{4t^{\frac{3}{2}}} d\xi \\
H(t) &= \frac{1}{\sqrt{t}}\Big[\frac{1}{\sqrt{\pi}} + \frac{\alpha \sqrt{t}}{4} e^{\frac{\alpha^2 t}{16}} \operatorname{\operatorname{erfc}}(-\frac{\alpha \sqrt{t}}{4})\Big].
\end{aligned}
\end{equation*}

Using the notation $\theta(\xi)=\Big(\phi_0(-\xi)+\phi_0(\xi)\Big)$ and applying integration by parts $F(t)$  can be written as 
\begin{equation*}
\begin{aligned}
   F(t) &= \frac{1}{\sqrt{\pi}} \int_0^{\infty}\frac{\xi e^{\frac{-\xi^2}{4t}}}{4t^{\frac{3}{2}}} \theta(\xi)\, d\xi \\
   &=\frac{1}{\sqrt{\pi}} \int_0^{\infty} e^{-y^2} \theta'(2\sqrt{t}y)\, dy.
\end{aligned}
\end{equation*}
If $\theta\in C^2(\mathbb{R})\cap W^{2,\infty}(\mathbb{R}),$  then the above expression for $F$ implies $F\in C^1 (0, \infty).$ Then the equation \eqref{convolution} implies $G\in C^1 (0, \infty)$.

\end{proof}

To obtain a weak solution of \eqref{IVP}, we consider the following two problems in the left and right half planes, respectively.

{\bf Problem 1:}
\begin{equation}\label{split problem left}
    \begin{aligned}
         U_t+\frac{U_x^2}{2}-U_{xx}&=0,~~~~~~~~ x<0, t>0\\
    U(x,0)&=\int_0^{x}u_0(y)dy\\
    U(0,t)&=g(t)
    \end{aligned}
\end{equation}

{\bf Problem 2:}

\begin{equation*}
    \begin{aligned}
        V_t+\frac{(V_x+\alpha)^2}{2}-V_{xx}&=0,~~~~~~~~ x>0, t>0\\
    V(x,0)&=\int_0^{x}u_0(y)dy\\
    V(0,t)&=g(t)
    \end{aligned}
\end{equation*}

Here, the function $g$ is unknown and can be specified by imposing a boundary regularity condition. The existence of $U$ and $V$ is given in the following theorem.

\begin{lemma}
\label{Hopf-Cole transformation}
Define $U=-2\log \phi$ and $V=-2\log \psi-\alpha x$ where $\phi$ and $\psi$ are given by \eqref{formula for phi} and \eqref{formula for psi} respectively. Then $U,V$ satisfy problem 1 and problem 2, respectively. Further $U\in C^{\infty}\big((-\infty, 0)\times (0, \infty)\big)$, $V\in C^{\infty} \big((0, \infty)\times (0, \infty)\big)$ and $U_x (0+, t)=V_x(0-,t).$
\end{lemma}

\begin{proof}
By lemma \eqref{solution for phi and psi} and  \eqref{properties of G}  $\phi, \psi>0$. If we take $U=-2\log \phi$ and $V=-2\log \psi-\alpha x,$ then $U$ satisfies problem 1 and $V$ satisfies problem 2.

    Observe that \begin{equation*}
    \begin{aligned}
        &U_x(0+,t)=-2\frac{\phi_x(0+,t)}{\phi(0+,t)}\\&
        V_x(0-,t)=-2\frac{\psi_x(0-,t)}{\psi(0-,t)}-\alpha.
    \end{aligned}
\end{equation*}
Therefore, $U_x(0+,t)=V_x(0-,t)$ if and only if $G(t)=\frac{2}{\alpha}\Big[\phi_x(0+,t)-\psi_x(0-,t)\Big].$ By the previous lemma, the result follows.

\end{proof}

\begin{lemma}
\label{regularity}
    If  the condition \eqref{extracondition} is satisfied then the explicit solution $u$ of \eqref{IVP} lies in the class  $C([0,\infty); H^1(\mathbb{R})\cap L^{\infty}(\mathbb{R})).$
\end{lemma}
\begin{proof}
 Define 
    \begin{equation*}
        u(x,t)=\begin{cases}
      U_x(x,t)&x<0, ~~t>0\\
      V_x(x,t)&x>0, ~~ t>0.
  \end{cases}
  \end{equation*}
    Using the definition of $u,U$ and $V$ from lemma \eqref{Hopf-Cole transformation} we get 
    $$u(x,t)=\begin{cases}
        -2\frac{\phi_x(x,t)}{\phi(x,t)},&x<0, ~~t>0\\
   -2\frac{\psi_x(x,t)}{\psi(x,t)}-\alpha, &x>0, ~~ t>0.
    \end{cases}$$
    We claim that $\phi(.,t):(-\infty,0)\to \mathbb{R}$ is bounded for fixed $t>0.$ Indeed, using change of variable in the first term of \eqref{left explicit formula} we get the following.
    \begin{equation*}
        \begin{aligned}
                |\phi(x,t)| &\leq \frac{M}{\sqrt{\pi}}\Big[\int_{-\infty}^{\frac{x}{2\sqrt{t}}}e^{-y^2}dy+\int^{\infty}_{\frac{x}{2\sqrt{t}}}e^{-y^2}dy \Big]\\&+\int_0^t|G^{\prime}(\tau)|\operatorname{erfc}\Big(\frac{-x}{2\sqrt{t-\tau}}\Big)d\tau+G(0)\operatorname{erfc}\Big(\frac{-x}{2\sqrt{t}}\Big)
        \end{aligned}
    \end{equation*} 
    for some constant $M.$ Since $G$ is a $C^1$ function, $\phi$ is bounded. Hence, it  suffices to show that $\phi_x(.,t),\phi_{xx}(.,t) \in L^2((-\infty,0)) .$ 
    From \eqref{definition 6} we can write the following :
    \begin{equation}\label{derivative of phi}
        \begin{aligned}
            \phi_x(x,t)&=  \frac{-1}{2\sqrt{\pi t}}\int_0^\infty \frac{\partial}{\partial x} \Big\{e^{-\frac{(x-\xi)^2}{4t}}-e^{-\frac{(x+\xi)^2}{4t}}\Big\}\phi_0(-\xi)d\xi\\&+\int_0^t\frac{G^{\prime}(\tau)}{\sqrt{\pi(t-\tau)}}e^{\frac{-x^2}{4(t-\tau)} }d\tau+\frac{G(0)}{\sqrt{\pi t}}e^{\frac{-x^2}{4t}}\\
    &= \frac{1}{2\sqrt{\pi t}}\int_0^\infty \frac{\partial}{\partial \xi} \Big\{e^{-\frac{(x-\xi)^2}{4t}}+e^{-\frac{(x+\xi)^2}{4t}}\Big\}\phi_0(-\xi)d\xi\\&+\int_0^t\frac{G^{\prime}(\tau)}{\sqrt{\pi(t-\tau)}}e^{\frac{-x^2}{4(t-\tau)} }d\tau+\frac{G(0)}{\sqrt{\pi t}}e^{\frac{-x^2}{4t}}.
        \end{aligned}
    \end{equation}
    Now  applying integration by parts to the first term of \eqref{derivative of phi} we get
 \begin{equation*}
     \begin{aligned}
             \phi_x(x,t)&= -\frac{1}{2\sqrt{\pi t}}\int_0^\infty \Big\{e^{-\frac{(x-\xi)^2}{4t}}+e^{-\frac{(x+\xi)^2}{4t}}\Big\}\frac{\partial}{\partial \xi} \Big(e^{-\frac{U(-\xi,0)}{2}}\Big)d\xi-\frac{1}{\sqrt{\pi t}}e^{-\frac{x^2}{4t}}\\&+\frac{x^2}{4\sqrt{\pi}}\int_0^te^{-\frac{x^2}{4 (t-\tau)}}\frac{G(\tau)}{(t-\tau)^{\frac{5}{2}}}d\tau-\frac{1}{2\sqrt{\pi}}\int_0^te^{-\frac{x^2}{4 (t-\tau)}}\frac{G(\tau)}{(t-\tau)^{\frac{3}{2}}}d\tau\\
             &= -\frac{1}{4\sqrt{\pi t}}\int_0^\infty \Big\{e^{-\frac{(x-\xi)^2}{4t}}+e^{-\frac{(x+\xi)^2}{4t}}\Big\}\phi_0(-\xi)u_0(-\xi)d\xi-\frac{1}{\sqrt{\pi t}}e^{-\frac{x^2}{4t}}\\&+\int_0^t\frac{G^{\prime}(\tau)}{\sqrt{\pi(t-\tau)}}e^{\frac{-x^2}{4(t-\tau)} }d\tau+\frac{G(0)}{\sqrt{\pi t}}e^{\frac{-x^2}{4t}}\\
             &=I_1+I_2+I_3+I_4.
     \end{aligned}
 \end{equation*}
 Applying the generalized Minkowski inequality, we obtain 
 \begin{equation*}
     \begin{aligned}
        & \Big(\int_{-\infty}^0  |\phi_x(x,t)|^2dx\Big)^{\frac{1}{2}}
       =\Big(\int_{-\infty}^0  |I_1+I_2+I_3+I_4|^2dx\Big)^{\frac{1}{2}}\\
       &\leq \Big[\int_{-\infty}^0\Big|I_1\Big|^2dx\Big]^{\frac{1}{2}}+\Big[\int_{-\infty}^0\Big|I_2\Big|^2dx\Big]^{\frac{1}{2}}+\Big[\int_{-\infty}^0\Big|I_3\Big|^2dx\Big]^{\frac{1}{2}}+\Big[\int_{-\infty}^0\Big|I_4\Big|^2dx\Big]^{\frac{1}{2}}\\
         &\leq\frac{1}{4\sqrt{\pi t}} \int_0^\infty \phi_0(-\xi)|u_0(-\xi) |\Big[\int_{-\infty}^0\Big|e^{-\frac{(x-\xi)^2}{4t}}+e^{-\frac{(x+\xi)^2}{4t}}\Big|^2 dx\Big]^{\frac{1}{2}}d\xi+\frac{1}{2^{\frac{1}{4}}}\\&+\int_0^t\frac{G^{\prime}(\tau)}{\sqrt{\pi(t-\tau)}} \Big[\int_{-\infty}^0e^{\frac{-x^2}{2(t-\tau)}}dx\Big]^{\frac{1}{2}}d\tau+\frac{G(0)}{2^{\frac{1}{4}}}
     \end{aligned}
 \end{equation*}
 Applying the change of variable in the first and third term of the above expression, we get 
\begin{equation*}
    \begin{aligned}
        &\Big(\int_{-\infty}^0  |\phi_x(x,t)|^2dx\Big)^{\frac{1}{2}}\\&\leq
        \frac{1}{4\sqrt{\pi t}} \int_0^\infty \phi_0(-\xi)|u_0(-\xi)| \Big[\int_{-\infty}^0\Big|e^{-\frac{(x-\xi)^2}{4t}}+e^{-\frac{(x+\xi)^2}{4t}}\Big|^2 dx\Big]^{\frac{1}{2}}d\xi+\frac{1}{2^{\frac{1}{4}}}\\&+\frac{1}{\sqrt{2}\pi ^{\frac{1}{4}}}\int_0^t\frac{G^{\prime}(\tau)}{(t-\tau)^{\frac{1}{4}}}d\tau+\frac{G(0)}{2^{\frac{1}{4}}}\\
        &\leq \frac{1}{(2\pi)^{\frac{3}{4}}t^{\frac{1}{4}}}\int_0^\infty\phi_0(-\xi)|u_0(-\xi)| d\xi+\frac{1}{2^{\frac{1}{4}}}+\frac{1}{\sqrt{2}\pi ^{\frac{1}{4}}}\int_0^t\frac{G^{\prime}(\tau)}{(t-\tau)^{\frac{1}{4}}}d\tau+\frac{G(0)}{2^{\frac{1}{4}}}
    \end{aligned}
\end{equation*} 
Since $G^{\prime}$ is bounded and by the condition on $u_0$ we obtain  $\phi_x(.,t)\in L^2((-\infty,0)).$
Similarly we can obtain  $\psi_x(.,t)\in L^2((0,\infty)).$
Hence from the definition of $u$ we get $u(.,t)\in L^2(\mathbb{R}).$ 

\noindent
Similarly, since $G^{\prime}$ is bounded and by the condition on $u_0$  we obtain  $\phi_{xx}(.,t)\in L^2((-\infty,0)) $ and $\psi_{xx}(.,t)\in L^2((0,\infty)).$
Hence, from the definition of $u$ we get $u_x(.,t)\in L^2(\mathbb{R}).$ 

\noindent
Therefore $u(.,t) \in H^1(\mathbb{R}).$ Since $\phi$ and $\phi_x$ are also bounded, thus we get $u(.,t)\in L^{\infty}(\mathbb{R}).$ Similarly using the generalized Minkowski inequality we can prove that $u$ is continuous. Therefore, the lemma follows.
\end{proof}

\subsubsection{Proof of the theorem \eqref{explicit solution}}
\begin{proof}
  From the lemma \eqref{regularity} $U\in C^{\infty}\big((-\infty, 0)\times (0, \infty)\big)$, $V\in C^{\infty} \big((0, \infty)\times (0, \infty)\big).$

    Let $\eta\in C_c^{\infty}(\mathbb{R}\times [0,\infty))$ and $\delta>0.$ Multiplying equation \eqref{split problem left}  by $\eta_x$ and integrating in the domain $(-\infty, -\delta]\times [0,\infty)$ we get the following .
    \begin{equation}
    \label{hamilton jacobi}
       \int_0^{\infty}\int_{-\infty}^{-\delta}\Big[U_t+\frac{U_x ^2}{2} -U_{xx}\Big]\eta_x dx dt=0
    \end{equation}
    Applying integration by parts, we get 
 
    \begin{equation}
    \label{ip}
        \begin{aligned}
        \int_0^{\infty}\int_{-\infty}^{-\delta} U_t \eta_x dx dt
        &= -\int_0^{\infty}\int_{-\infty}^{-\delta} U\eta_{tx} dx dt-\int_{-\infty}^{-\delta} U(x,0)\eta_x (x.0) dx \\
        &=\int_0^{\infty}\int_{-\infty}^{-\delta} U_x \eta_{t} dx dt-\int_0^{\infty}U(-\delta, t)\eta_t(-\delta, t) dt -\int_{-\infty}^{-\delta} U(x,0)\eta_x(x,0) dx.
        \end{aligned}
    \end{equation}
From \eqref{hamilton jacobi} and \eqref{ip}, we get
     \begin{equation*}
        \begin{aligned}
        & \int_0^{\infty}\int_{-\infty}^{-\delta} [U_x \eta_{t} + \frac{U_x ^2}{2} \eta_x -U_{xx} \eta_x] dx dt -\int_0^{\infty}U(-\delta, t)\eta_t(-\delta, t) dt\\
        &-\int_{-\infty}^{-\delta} U(x,0)\eta_x(x,0) dx=0.
        \end{aligned}
    \end{equation*}
In view of the above lemma \eqref{regularity}, passing to the limit as $\delta\to 0,$ we get 

 \begin{equation}\label{IP1}
        \begin{aligned}
        & \int_0^{\infty}\int_{-\infty}^{0} [U_x \eta_{t} + \frac{U_x ^2}{2} \eta_x -U_{xx} \eta_x] dx dt -\int_0^{\infty}U(0-, t)\eta_t(0, t) dt\\
        &-\int_{-\infty}^{0} U(x,0)\eta_x(x,0) dx=0.
        \end{aligned}
    \end{equation}

     Similarly, we get 

    \begin{equation}\label{IP2}
        \begin{aligned}
        & \int_0^{\infty}\int_{0}^{\infty} [V_x \eta_{t} + \frac{(V_x+\alpha)^2}{2} \eta_x -V_{xx} \eta_x] dx dt +\int_0^{\infty}V(0+, t)\eta_t(0, t) dt\\
        &-\int_{0}^{\infty} V(x,0)\eta_x(x,0) dx=0.
        \end{aligned}
    \end{equation}
    Adding \eqref{IP1} and \eqref{IP2}, we get the weak formulation \eqref{integral identity}. Putting back the value of $U$ and $V$ from the lemma \eqref{Hopf-Cole transformation} and value of $G$ from the lemma \eqref{properties of G} we get the explicit formula \eqref{explicit formula}. This completes the proof.
\end{proof}

\subsection{Uniqueness of weak solution}
\begin{theorem} 
\label{uniqueness theorem}
The weak solution of \eqref{IVP} in the sense of definition\eqref{integral identity} is unique.
\end{theorem}
\begin{proof}
  To prove it, we follow the argument given in \cite{chung}. Let $u,v$ be two weak solutions with initial data $u_0.$ Define $e:=u-v.$ Then $e$ satisfies the following equation
    \begin{equation}
    \label{integral identity 1}
        \begin{aligned}
              \int_{0}^{\infty} \int_{-\infty}^{\infty}\Big[ e\zeta_t + [(1-H(x))\frac{e(u+v)}{2}+H(x)\frac{e(u+v+2\alpha)}{2}-e_x] \zeta_x\Big] dx dt=0
        \end{aligned}
    \end{equation} 
    for all test functions $\zeta \in C_c^{\infty}(\mathbb{R}\times [0,\infty)).$ Let $\epsilon >0$ and $T>0$ be fixed.
    By density argument, we can choose $\zeta$  as $\zeta(x,t)= \beta^{\epsilon}(t)e(x,t)$ where 
    \begin{equation*}
      \beta^{\epsilon} (t)=\begin{cases}
          1, &0\leq t < T-\epsilon\\
          \frac{T-t}{\epsilon}, &T-\epsilon \leq t< T\\
          0, &t>T.
      \end{cases}  
    \end{equation*}
In addition to this, we can assume $e\in  C^1([0,\infty); H^1(\mathbb{R})).$ Later by approximation it follows for
$e\in  C([0,\infty); H^1(\mathbb{R})).$ 
     Therefore, passing to the limit as  $\epsilon\to 0$(details given below),  we obtain 

    \begin{equation}
    \label{passage to the limit}
     \begin{aligned}
        \int_{0}^{\infty} \int_{-\infty}^{\infty} e\zeta_t dx dt
    &=\int_{0}^{\infty} \int_{-\infty}^{\infty}  e(x,t).e_t(x,t) \beta^{\epsilon}(t) dx dt + \int_{0}^{\infty}{\beta^{\epsilon}}^{\prime}(t)\int_{-\infty}^{\infty}  e^2 (x,t)dx  dt\\ 
     &= \int_{0}^{\infty}  \beta^{\epsilon}(t) <e, e_t> dt -\frac{1}{\epsilon} \int_{T-\epsilon}^{T} \|e\|^2 dt\\
     &\to \frac{1}{2}\int_{0}^T \frac{d}{dt} ( \|e(., t)\|^2)   dt - \|e(.,T)\|^2 =-\frac{1}{2}  \|e(.,T)\|^2\\
     \int_{0}^{\infty} \int_{-\infty}^{\infty} e_x\zeta_x dx dt &\to \int_0^T||e_x(t)||_2^2 dt. 
    \end{aligned}
 \end{equation}
    
     Hence, from \eqref{passage to the limit} and \eqref{integral identity 1} we get the following.
     \begin{equation*}
         \begin{aligned}
             ||e(T)||_2^2+2\int_0^T||e_x(t)||_2^2 dt&= 2 \int_{0}^{T} \int_{-\infty}^{\infty} \Big[(1-H(x))\frac{e(u+v)}{2}+H(x)\frac{e(u+v+2\alpha)}{2}\Big]e_x dx dt.
         \end{aligned}
     \end{equation*}
 Since $u, v\in C([0,\infty); H^1(\mathbb{R})\cap L^{\infty}(\mathbb{R})),$ we get  
 $$||e(T)||_2^2+2\int_0^T||e_x(t)||_2^2 dt\leq M \int_0^T \int_{-\infty}^{\infty}| ee_x| dx  dt,$$
 for some positive constant $M.$ Thus, using  the H$\ddot{o}$lder inequality and Cauchy's inequality with $\epsilon$ we get  \begin{equation*}
         \begin{aligned}
              ||e(T)||_2^2+2\int_0^T||e_x(t)||_2^2 dt&
              \leq  M \int_0^T \Big(\epsilon ||e_x(t)||_2^2+\frac{1}{4\epsilon} ||e(T)||_2^2\Big)dt\\&=M\epsilon \int_0^T ||e_x(t)||_2^2 dt + \frac{M}{4\epsilon}  \int_0^T||e(T)||_2^2dt
         \end{aligned}
     \end{equation*} 
     which implies  
     \begin{equation*}
         \begin{aligned}
             ||e(T)||_2^2+(2-M\epsilon)\int_0^T||e_x(t)||_2^2 dt\leq\frac{M}{4\epsilon}  \int_0^T||e(T)||_2^2dt
         \end{aligned}
     \end{equation*} 
     Choose $\epsilon$ so that $(2-M\epsilon )>0.$ Then from the last inequality we get
     \begin{equation*}
          ||e(T)||_2^2\leq\frac{M}{4\epsilon}  \int_0^T||e(T)||_2^2dt.
     \end{equation*}
     Hence, from Gronwall's inequality, we get $ ||e(T)||_2^2=0.$ Since $T>0$ is arbitrary, we get $||e(t)||_2^2=0$ for all $t>0.$ This proves the uniqueness of the weak solution.
\end{proof}

\section{Large time behavior of the solution for $\alpha<0$}
In this section, under the assumption \eqref{extracondition}, we investigate the large-time behavior of the solution for $\alpha<0$. First, we obtain the large time behavior of the function $G.$
\subsection{Asymptotic behavior of the intermediate function $G$}
We assume the following conditions along with $\alpha<0$


The boundary condition $G(t)$ satisfies an asymptotic behavior, which is given by the following lemma.
\begin{lemma}\label{approximation of G}
   Under assumption \eqref{extracondition} along with $\alpha<0,$  $G(t)$ obtained in lemma \eqref{properties of G}  satisfies the following: $$\displaystyle{\lim_{t \to \infty}}\sqrt{t}G(t)=-\frac{2}{\alpha\sqrt{\pi}}\Big[\phi_0(-\infty)+\phi_0(\infty)\Big].$$
\end{lemma}
\begin{proof}
Recall the formula of $G$ as in  equation\eqref{formula for G}.

\begin{equation*}
G(t)=\frac{1}{\sqrt{\pi}}\int_0^t\int_0^{\infty}\frac{\xi e^{\frac{-\xi^2}{4s}}\theta(\xi)}{4(s)^{\frac{3}{2}}\sqrt{t-s}}\Big[\frac{1}{\sqrt{\pi}}+\frac{\alpha\sqrt{t-s}}{4}e^{\frac{\alpha^2(t-s)}{16}}\operatorname{erfc}\Big(\frac{-\alpha\sqrt{t-s}}{4}\Big)\Big]d\xi ds.
\end{equation*}
Define $H(x):=\frac{1}{\sqrt{\pi}}-xe^{x^2}\operatorname{erfc}(x)$ and $f_1(t)=\int_0^{\infty}\frac{ue^{-u^2}\theta(2u\sqrt{t})}{\sqrt{t}}du.$  Then with these notations, the  equation \eqref{formula for G} can be simplified as follows.
    \begin{equation}\label{expression of G(t)}
        \begin{aligned}
            G(t)&=\frac{1}{\sqrt{\pi}}\int_0^t\int_0^{\infty}\frac{\xi e^{\frac{-\xi^2}{4s}}\theta(\xi)}{4(s)^{\frac{3}{2}}\sqrt{t-s}}H\Big(\frac{-\alpha\sqrt{t-s}}{4}\Big)d\xi ds\\&=\frac{1}{\sqrt{\pi}}\int_0^t\int_0^{\infty}\frac{ue^{-u^2}}{\sqrt{s(t-s)}}\theta(2u\sqrt{s})H\Big(\frac{-\alpha\sqrt{t-s}}{4}\Big)duds\\&=\frac{1}{\sqrt{\pi}}\int_0^t\frac{H\Big(\frac{-\alpha\sqrt{t-s}}{4}\Big)}{\sqrt{t-s}}\Big(\int_0^{\infty}\frac{ue^{-u^2}\theta(2u\sqrt{s})}{\sqrt{s}}du\Big)ds\\&=\frac{1}{\sqrt{\pi}}\int_0^t\frac{H\Big(\frac{-\alpha\sqrt{s}}{4}\Big)}{\sqrt{s}}f_1(t-s)ds.
        \end{aligned}
    \end{equation}  

    In order to determine the limit of $\sqrt{t} G(t)$, we use the following decomposition from \eqref{expression of G(t)}.
    
    \begin{equation}\label{equality 2}
        \begin{aligned}
            \sqrt{t}G(t)&=\frac{1}{\sqrt{\pi}}\int_0^t\frac{H\Big(\frac{-\alpha\sqrt{s}}{4}\Big)}{\sqrt{s}}(\sqrt{t}-\sqrt{t-s})f_1(t-s)ds+\frac{1}{\sqrt{\pi}}\int_0^t\frac{H\Big(\frac{-\alpha\sqrt{s}}{4}\Big)}{\sqrt{s}}\sqrt{t-s}f_1(t-s)ds\\&=A+B.
        \end{aligned}
    \end{equation}
    First consider $A.$  The expression for $A$ can be rewritten

    \begin{equation*}
        \begin{aligned}
            A=\frac{1}{\sqrt{\pi}}\int_{0}^{\infty} \chi_{[0,t]}(s)(\sqrt{t}-\sqrt{t-s})f_1(t-s)\frac{H\Big(\frac{-\alpha\sqrt{s}}{4}\Big)}{\sqrt{s}} ds.
        \end{aligned}
    \end{equation*}
    
    Notice that 
    \begin{equation}
    \label{essential inequalities}
    \begin{aligned}
       &\sqrt{t}-\sqrt{t-s}=\frac{s}{\sqrt{t}+\sqrt{t-s}}\leq \frac{s}{\sqrt{t}} \\
       & |H\Big(\frac{-\alpha\sqrt{s}}{4}\Big)| \leq \frac{C_1} {1+ \frac{\alpha^2 s}{16}}\\
       & |f_1|\leq C_2,~~~ \lim_{t\to \infty} f_1 (t)=0,
    \end{aligned}
    \end{equation}
for some positive constants $C_1$ and $C_2.$ The last two statements are followed by the observation 
    
\begin{equation*}
        \begin{aligned}
            f_1(t)=\int_0^{\infty}\frac{ue^{-u^2}\theta(2u\sqrt{t})}{\sqrt{t}}du=\int_0^{\infty}e^{-u^2}\theta^{\prime}(2u\sqrt{t})du.
        \end{aligned}
    \end{equation*} 
Now, applying the above inequalities \eqref{essential inequalities}

     \begin{equation}
     \label{ess-1}
        \begin{aligned}
            |A|&\leq \frac{1}{\sqrt{\pi}} \int_{0}^{t} |f_1 (t-s)| \frac{C_1 \sqrt{s}} {\sqrt{t}(1+ \frac{\alpha^2 s}{16})} ds\\
               &\leq \frac{C_1}{\sqrt{\pi t}}\int_{0}^{t} |f_1 (s)| \frac{\sqrt{t-s}} {(1+ \frac{\alpha^2 (t-s)}{16})} ds\\
             &  \leq \frac{4C_1}{\alpha\sqrt{\pi t}} \int_{0}^{t} \frac{|f_1(s)|}{\sqrt{1+ \frac{\alpha^2 (t-s)}{16}}} ds.
        \end{aligned}
    \end{equation}
Since $f_1(t)\to 0,$ given an $\epsilon >0,$ there exists a $T$  such that $|f_1 (t)| \leq \epsilon$ for $t\geq T.$

Now split the above integrals into two parts:
\begin{equation}
\label{ess-2}
\begin{aligned}
    \int_{0}^{t} \frac{|f_1(s)|}{\sqrt{1+ \frac{\alpha^2 (t-s)}{16}}} ds= \int_{0}^{T} \frac{|f_1(s)|}{\sqrt{1+ \frac{\alpha^2 (t-s)}{16}}} ds+\int_{T}^t \frac{|f_1(s)|}{\sqrt{1+ \frac{\alpha^2 (t-s)}{16}}} ds.
\end{aligned}
    \end{equation}

Now 

\begin{equation}
\label{ess-3}
\begin{aligned}
\int_{0}^{T} \frac{|f_1(s)|}{\sqrt{1+ \frac{\alpha^2 (t-s)}{16}}} ds&\leq C_2 \int_{0}^{T} \frac{1}{\sqrt{1+ \frac{\alpha^2 (t-s)}{16}}} ds\\
\int_{T}^{t} \frac{|f_1(s)|}{\sqrt{1+ \frac{\alpha^2 (t-s)}{16}}} ds&\leq \epsilon \int_{T}^{t} \frac{1}{\sqrt{1+ \frac{\alpha^2 (t-s)}{16}}} ds\\
&\leq \epsilon \int_{0}^{t} \frac{1}{\sqrt{1+ \frac{\alpha^2 (t-s)}{16}}} ds=\frac{32 \epsilon }{\alpha^2}
\left(
\sqrt{1+\frac{\alpha^2 t}{16}}-1
\right).
\end{aligned}
    \end{equation}
From equations \eqref{ess-1},\eqref{ess-2}, \eqref{ess-3}, it follows that $|A| \to 0$ as $t\to \infty.$
    
    Next, consider $B=\frac{1}{\sqrt{\pi}}\int_0^{\infty}\chi_{[0,t]}\frac{H\Big(\frac{-\alpha\sqrt{s}}{4}\Big)}{\sqrt{s}}\sqrt{t-s}f_1(t-s)ds.$

\noindent
    
    Now $$\sqrt{t}f_1(t)=\int_0^{\infty}ue^{-u^2}\Big(\phi_0(-2u\sqrt{t})+\phi_0(2u\sqrt{t})\Big)du.$$

    \noindent
    Assume the condition \eqref{extracondition} .
    Now from \eqref{derivative of h} $h^{\prime}(s)=\frac{\alpha}{4\sqrt{s}}H\Big(\frac{-\alpha\sqrt{s}}{4}\Big),$ so  
\begin{equation*}
        \begin{aligned}
            \int_0^{\infty}\frac{H\Big(\frac{-\alpha\sqrt{s}}{4}\Big)}{\sqrt{s}}ds=\frac{4}{\alpha}\int_0^{\infty}h^{\prime}(s)ds=-\frac{4}{\alpha}.
        \end{aligned}
    \end{equation*} 
    Then, by the dominated convergence theorem, we get $$\displaystyle{\lim_{t\to \infty}}\sqrt{t}f_1(t)=\frac{1}{2}\Big[\phi_0(-\infty)+\phi_0(\infty)\Big]$$ and \begin{equation*}
        \begin{aligned}
            \displaystyle{\lim_{t\to \infty}}B&=\frac{\Big[\phi_0(-\infty)+\phi_0(\infty)\Big]}{2\sqrt{\pi}}\int_0^{\infty}\frac{H\Big(\frac{-\alpha\sqrt{s}}{4}\Big)}{\sqrt{s}}ds\\&=-\frac{2}{\alpha\sqrt{\pi}}\Big[\phi_0(-\infty)+\phi_0(\infty)\Big].
        \end{aligned}
    \end{equation*} Thus, from \eqref{equality 2} we get $$\displaystyle{\lim_{t\to \infty}}\sqrt{t}G(t)=-\frac{2}{\alpha\sqrt{\pi}}\Big[\phi_0(-\infty)+\phi_0(\infty)\Big].$$
\end{proof}
\subsection{Asymptotic behavior of the weak solution}
In this section, we derive the asymptotic behavior of this solution using the asymptotic limit of $\sqrt{t}G(t)$. In order to obtain the asymptotic behavior of the solution, we first obtain the asymptotic behavior of $\phi,\psi$ and $\phi_x,\psi_x$ separately as in the following two lemmas.
\begin{lemma}\label{approximation of phi}
      Under assumption \eqref{extracondition} along with $\alpha<0,$ $\sqrt{t}\phi$ and $\sqrt{t}\phi_x$ converge uniformly on compact subsets of $(-\infty, 0)$,  and the limits are given by \begin{equation*}
        \begin{aligned}
            &\displaystyle{\lim_{t \to \infty}}\sqrt{t}\phi(x,t)=-\frac{\phi_0(-\infty)}{\sqrt{\pi}}x-\frac{2}{\alpha\sqrt{\pi}}\Big[\phi_0(-\infty)+\phi_0(\infty)\Big],\\&
            \displaystyle{\lim_{t \to \infty}}\sqrt{t}\phi_x(x,t)=-\frac{\phi_0(-\infty)}{\sqrt{\pi}}.
        \end{aligned}
    \end{equation*} 
\end{lemma}
\begin{proof}
From \eqref{formula for phi} we get 
\begin{equation*}
\begin{aligned}
    \sqrt{t}\phi(x,t)&=\frac{-1}{2\sqrt{\pi}}\int_0^\infty \Big\{e^{-\frac{(x-\xi)^2}{4t}}-e^{-\frac{(x+\xi)^2}{4t}}\Big\}\phi_0(-\xi)d\xi\\&-\frac{x}{2\sqrt{\pi}}\int_0^t\frac{e^{-\frac{x^2}{4 \tau}}}{\tau^{\frac{3}{2}}}(\sqrt{t}-\sqrt{t-\tau})G(t-\tau)d\tau-\frac{x}{2\sqrt{\pi}}\int_0^t\frac{e^{-\frac{x^2}{4 \tau}}}{\tau^{\frac{3}{2}}}\sqrt{t-\tau}G(t-\tau)d\tau\\&=I_3+I_4+I_5
\end{aligned}
\end{equation*} where \begin{equation*}
    \begin{aligned}
    &I_3=\frac{-1}{2\sqrt{\pi}}\int_0^\infty \Big\{e^{-\frac{(x-\xi)^2}{4t}}-e^{-\frac{(x+\xi)^2}{4t}}\Big\}\phi_0(-\xi)d\xi,
    \\&I_4=-\frac{x}{2\sqrt{\pi}}\int_0^{\infty}\frac{e^{-\frac{x^2}{4 \tau}}}{\tau^{\frac{3}{2}}}\chi_{[0,t]}(\sqrt{t}-\sqrt{t-\tau})G(t-\tau)d\tau\\&I_5=-\frac{x}{2\sqrt{\pi}}\int_0^\infty \frac{e^{-\frac{x^2}{4 \tau}}}{\tau^{\frac{3}{2}}}\chi_{[0,t]}\sqrt{t-\tau}G(t-\tau)d\tau.
    \end{aligned}
\end{equation*} 

 Now, we make the simplification of $I_3$ using a change of variable as follows:

\begin{equation*}
 \begin{aligned}
    -2\sqrt{\pi} I_3&= e^{-\frac{x^2}{4t}}\int_{0}^{\infty} e^{-\frac{\xi^2}{4t}}\Big[ e^{\frac{x \xi}{2t}}- e^{-\frac{x \xi}{2t}}\Big] \phi_0(-\xi) d \xi\\
     &=2e^{-\frac{x^2}{4t}}\sqrt{t}\int_{0}^{\infty} e^{-y^2}\Big[ e^{\frac{xy}{\sqrt{t}}}- e^{-\frac{xy}{\sqrt{t}}}\Big] \phi_0(-2 \sqrt{t}y) d y\\
     &=e^{-\frac{x^2}{4t}}\Big[2\sqrt{t}\int_{0}^{T} e^{-y^2}\big( e^{\frac{xy}{\sqrt{t}}}- e^{-\frac{xy}{\sqrt{t}}}\big) \phi_0(-2 \sqrt{t}y) d y\\&+2\sqrt{t}\int_{T}^{\infty} e^{-y^2}\big( e^{\frac{xy}{\sqrt{t}}}- e^{-\frac{xy}{\sqrt{t}}}\big) \phi_0(-2 \sqrt{t}y) d y\Big]\\
     &= e^{-\frac{x^2}{4t}}[I_{31}+I_{32}],
 \end{aligned}   
\end{equation*}
where $T>0$ is any positive number. 
Now 
\begin{equation*}
 \begin{aligned}
 I_{31}&=2\sqrt{t}\int_{0}^{T} e^{-y^2}\big( e^{\frac{xy}{\sqrt{t}}}- e^{-\frac{xy}{\sqrt{t}}}\big) \phi_0(-2 \sqrt{t}y) d y\\
 &=2\int_{0}^{T} e^{-y^2}\big( \int_{-1}^{1} xye^{\frac{sxy}{\sqrt{t}}} ds\big) \phi_0(-2 \sqrt{t}y) d y\\
 &=2x \int_{-1}^{1} \int_{0}^{T} ye^{-y^2}e^{\frac{sxy}{\sqrt{t}}} \phi_0(-2 \sqrt{t}y) d y ds
 \end{aligned}   
\end{equation*}

Using the dominated convergence theorem, we get 
\begin{equation}
\label{limit I_31}
\lim_{t\to \infty} I_{31}=2x \phi_0(-\infty)\int_{0}^{T}y e^{- y^2} dy.
\end{equation}

For large enough $t,$ the second integral $I_{32}$ becomes
 \begin{equation}
 \label{limit I_{32}}
 \begin{aligned}
  I_{32}&=2\sqrt{t}\int_{T}^{\infty} e^{-y^2}\big( e^{\frac{xy}{\sqrt{t}}}- e^{-\frac{xy}{\sqrt{t}}}\big) \phi_0(-2 \sqrt{t}y) d y\\
  &=\phi_0(-\infty)(1+o(t))2\sqrt{t}\int_{T}^{\infty} e^{-y^2}\big( e^{\frac{xy}{\sqrt{t}}}- e^{-\frac{xy}{\sqrt{t}}}\big)  d y\\
  &= \phi_0(-\infty)(1+o(t))\int_{2 \sqrt{t}T}^{\infty}  \Big\{e^{-\frac{(x-\xi)^2}{4t}}-e^{-\frac{(x+\xi)^2}{4t}}\Big\} d \xi\\
  &=\phi_0(-\infty)(1+o(t)) 2\sqrt{t}\Big[ \int_{-\infty}^{\frac{x}{2 \sqrt{t}}-T} e^{- u^2} du-\int_{\frac{x}{2 \sqrt{t}}+T}^{\infty} e^{-u^2} du\Big]
\end{aligned}   
\end{equation}

Applying L'Hospital's rule we obtain,

 \begin{equation*}
 \label{Lopital rule}
 \begin{aligned}
& \lim_{t\to \infty} 2\sqrt{t}\Big[ \int_{-\infty}^{\frac{x}{2 \sqrt{t}}-T} e^{- u^2} du-\int_{\frac{x}{2 \sqrt{t}}+T}^{\infty} e^{-u^2} du\Big]\\
 =&2 \lim_{t\to \infty} \dfrac{ \int_{-\infty}^{\frac{x}{2 \sqrt{t}}-T} e^{- u^2} du-\int_{\frac{x}{2 \sqrt{t}}+T}^{\infty}e^{-u^2}du}{t^{-\frac{1}{2}}}\\
 =&2 \lim_{t\to \infty} \dfrac{-\frac{x}{4t^{3/2}}\Big[e^{-(\frac{x}{2\sqrt t}-T)^2}
+e^{-\left(\frac{x}{2\sqrt t}+T\right)^2}\Big]
.}{-\frac{1}{2}t^{-\frac{3}{2}}}\\
&=2 x e^{-T^2}
\end{aligned}   
\end{equation*}

Since $T$ arbritary, Combining \eqref{limit I_31}, \eqref{limit I_{32}}, we get

$$\lim_{t\to \infty} I_3= -\frac{\phi_0(-\infty)}{\sqrt{\pi}}x$$

  Notice that 
    \begin{equation}
    \label{essential inequalities 1}
    \begin{aligned}
       &\sqrt{t}-\sqrt{t-\tau}=\frac{\tau}{\sqrt{t}+\sqrt{t-\tau}}\leq \frac{\tau}{\sqrt{t}} \\
       & |G|\leq C_3,~~~ \lim_{t\to \infty} G (t)=0,
    \end{aligned}
    \end{equation}
for some positive constant  $C_3.$ The last statement follows from lemma \eqref{approximation of G}.
Now applying the above inequalities \eqref{essential inequalities  1} we have
\begin{equation}
\label{ess-1'}
    \begin{aligned}
        |I_4|&\leq-\frac{x}{2\sqrt{\pi}}\int_0^{t}\frac{e^{-\frac{x^2}{4 (t-\tau)}}}{(t-\tau)^{\frac{3}{2}}}\frac{t-\tau}{\sqrt{t}} G(\tau)d\tau\\
        & \leq -\frac{x}{2\sqrt{\pi t}}\int_0^{t}\frac{e^{-\frac{x^2}{4 (t-\tau)}}}{\sqrt{(t-\tau)}} |G(\tau)|d\tau.
    \end{aligned}
\end{equation} 
Since $G(t)\to 0$ as $t \to \infty,$ given an $\epsilon >0,$ there exists a $T$  such that $|G (t)| \leq \epsilon$ for $t\geq T.$

\noindent
Now split the above integrals into two parts:
\begin{equation}
\label{ess-2'}
\begin{aligned}
   \int_0^{t}\frac{e^{-\frac{x^2}{4 (t-\tau)}}}{\sqrt{(t-\tau)}} |G(\tau)|d\tau= \int_{0}^{T} \frac{e^{-\frac{x^2}{4 (t-\tau)}}}{\sqrt{(t-\tau)}} |G(\tau)|d\tau+\int_{T}^t \frac{e^{-\frac{x^2}{4 (t-\tau)}}}{\sqrt{(t-\tau)}} |G(\tau)|d\tau.
\end{aligned}
    \end{equation}

Now 

\begin{equation}
\label{ess-3'}
\begin{aligned}
\int_{0}^{T}  \frac{e^{-\frac{x^2}{4 (t-\tau)}}}{\sqrt{(t-\tau)}} |G(\tau)|d\tau&\leq C_3 \int_{0}^{T} \frac{e^{-\frac{x^2}{4 (t-\tau)}}}{\sqrt{(t-\tau)}} d\tau\\
\int_{T}^{t}\frac{e^{-\frac{x^2}{4 (t-\tau)}}}{\sqrt{(t-\tau)}} |G(\tau)|d\tau&\leq \epsilon \int_{T}^{t}\frac{e^{-\frac{x^2}{4 (t-\tau)}}}{\sqrt{(t-\tau)}} d\tau\\
&\leq \epsilon \int_{0}^{t} \frac{e^{-\frac{x^2}{4 (t-\tau)}}}{\sqrt{(t-\tau)}} d\tau\leq \epsilon \int_{0}^{t} \frac{1}{\sqrt{(t-\tau)}} d\tau=2\epsilon \sqrt{t}.
\end{aligned}
    \end{equation}
From equations \eqref{ess-1'},\eqref{ess-2'}, \eqref{ess-3'}, it follows that $|I_4| \to 0$ as $t\to \infty.$ Next we consider $I_5.$ By the dominated convergence theorem and lemma \eqref{approximation of G} we get the following limit.
\begin{equation*}
    \begin{aligned}
        & \displaystyle{\lim_{t \to \infty}}I_5=\frac{x}{\alpha{\pi}}\Big[\phi_0(-\infty)+\phi_0(\infty)\Big]\int_0^{\infty}\frac{e^{-\frac{x^2}{4 \tau}}}{\tau^{\frac{3}{2}}}d\tau\\&=-\frac{2}{\alpha\sqrt{\pi}}\Big[\phi_0(-\infty)+\phi_0(\infty)\Big].
    \end{aligned}
\end{equation*}
Therefore, $$\displaystyle{\lim_{t \to \infty}}\sqrt{t}\phi(x,t)=-\frac{\phi_0(-\infty)}{\sqrt{\pi}}x+\frac{2}{\alpha\sqrt{\pi}}\Big[\phi_0(-\infty)+\phi_0(\infty)\Big].$$
Next consider $\phi_x.$ From \eqref{definition 6} we get \begin{equation*}
    \begin{aligned}
        \phi_x(x,t)&=\frac{-1}{4\sqrt{\pi}(t)^{\frac{3}{2}}}\int_0^{\infty}\phi_0(-\xi)\Big\{(\xi-x)e^{-\frac{(x-\xi)^2}{4t}}+(x+\xi)e^{-\frac{(x+\xi)^2}{4t}}\Big\}d\xi\\&+\int_0^t\frac{e^{\frac{-x^2}{4\tau} }}{\sqrt{\pi\tau}}G^{\prime}(t-\tau)d\tau+\frac{G(0)}{\sqrt{\pi t}}e^{\frac{-x^2}{4t}}.
    \end{aligned}
\end{equation*}
Changing the variable in the first integral and applying the integration by parts in the second integral, we obtain \begin{equation}\label{equality 5}
    \begin{aligned}
         \phi_x(x,t)&=\frac{1}{\sqrt{\pi t}}\Big[\int_{-\infty}^{\frac{x}{2\sqrt{t}}}ue^{-u^2}\phi_0(2u\sqrt{t}-x)du-\int_{\frac{x}{2\sqrt{t}}}^{\infty}ve^{-v^2}\phi_0(x-2v\sqrt{t})dv\Big]\\&+\frac{x^2}{4\sqrt{\pi}}\int_0^t\frac{e^{-\frac{x^2}{4\tau}}}{\tau^{\frac{5}{2}}}G(t-\tau)d\tau-\frac{1}{2\sqrt{\pi}}\int_0^t\frac{e^{-\frac{x^2}{4\tau}}}{\tau^{\frac{3}{2}}}G(t-\tau)d\tau-\frac{G(0)}{\sqrt{\pi t}}e^{\frac{-x^2}{4t}}+\frac{G(0)}{\sqrt{\pi t}}e^{\frac{-x^2}{4t}}\\&=J_1+J_2+J_3
    \end{aligned}
\end{equation} where \begin{equation*}
    \begin{aligned}
        &J_1=\frac{1}{\sqrt{\pi t}}\Big[\int_{-\infty}^{0}ue^{-u^2}\Big[\chi_{(-\infty,\frac{x}{2\sqrt{t}}]}\phi_0(2u\sqrt{t}-x)\Big]du-\int_{0}^{\infty}ve^{-v^2}\Big[\chi_{[\frac{x}{2\sqrt{t}},\infty)}\phi_0(x-2v\sqrt{t})\Big]dv\Big],\\&
        J_2=\frac{x^2}{4\sqrt{\pi}}\int_0^{\infty}\frac{e^{-\frac{x^2}{4\tau}}}{\tau^{\frac{5}{2}}}\chi_{[0,t]}G(t-\tau)d\tau,\\&J_3=-\frac{1}{2\sqrt{\pi}}\int_0^{\infty}\frac{e^{-\frac{x^2}{4\tau}}}{\tau^{\frac{3}{2}}}\chi_{[0,t]}G(t-\tau)d\tau.
    \end{aligned}
\end{equation*}  Then by the dominated convergence theorem we get \begin{equation*}
    \begin{aligned}
        \displaystyle{\lim_{t \to \infty}}\sqrt{t}J_1&=\frac{\phi_0(-\infty)}{\sqrt{\pi}}\Big[\int_{-\infty}^0ue^{-u^2}du-\int_{0}^{\infty}ve^{-v^2}dv\Big]\\&=-\frac{\phi_0(-\infty)}{\sqrt{\pi}}.
    \end{aligned}
\end{equation*}
Now we have \begin{equation*}
    \begin{aligned}
        \sqrt{t}J_2&=\frac{x^2}{4\sqrt{\pi}}\int_0^{\infty}\frac{e^{-\frac{x^2}{4\tau}}}{\tau^{\frac{5}{2}}}\chi_{[0,t]}(\sqrt{t}-\sqrt{t-\tau})G(t-\tau)d\tau\\&+\frac{x^2}{4\sqrt{\pi}}\int_0^{\infty}\frac{e^{-\frac{x^2}{4\tau}}}{\tau^{\frac{5}{2}}}\chi_{[0,t]}\sqrt{t-\tau}G(t-\tau)d\tau=J_5+J_6.
    \end{aligned}
\end{equation*} By the similar calculation to $I_4$ we get $\displaystyle{\lim_{t \to \infty}}|J_5|=0.$
 By  the lemma \eqref{approximation of G} and the  dominated convergence theorem we get
 \begin{equation*}
    \begin{aligned}
        \displaystyle{\lim_{t \to \infty}}J_6&
        =-\frac{2}{\alpha\sqrt{\pi}x}\Big[\phi_0(-\infty)+\phi_0(\infty)\Big].
    \end{aligned}
\end{equation*}
Next we have \begin{equation*}
    \begin{aligned}
         \sqrt{t}J_3&=-\frac{1}{2\sqrt{\pi}}\int_0^{\infty}\frac{e^{-\frac{x^2}{4\tau}}}{\tau^{\frac{3}{2}}}\chi_{[0,t]}(\sqrt{t}-\sqrt{t-\tau})G(t-\tau)d\tau\\&-\frac{1}{2\sqrt{\pi}}\int_0^{\infty}\frac{e^{-\frac{x^2}{4\tau}}}{\tau^{\frac{3}{2}}}\chi_{[0,t]}\sqrt{t-\tau}G(t-\tau)d\tau=J_7+J_8.
    \end{aligned}
\end{equation*} Therefore, as above we have $\displaystyle{\lim_{t \to \infty}}|J_7|=0$ and 
\begin{equation*}
\begin{aligned}
    \displaystyle{\lim_{t \to \infty}}J_8&
    =\frac{2}{\alpha x \sqrt{\pi}}\Big[\phi_0(-\infty)+\phi_0(\infty)\Big].
\end{aligned}
\end{equation*}
 Hence, from \eqref{equality 5} we get \begin{equation*}
    \begin{aligned}
        \displaystyle{\lim_{t \to \infty}}\sqrt{t}\phi_x(x,t)&
        =-\frac{\phi_0(-\infty)}{\sqrt{\pi}}.
    \end{aligned}
\end{equation*}
\end{proof} 
Similarly, we can prove the following lemma.
\begin{lemma}\label{approximation of psi}
   Under assumption \eqref{extracondition} along with $\alpha<0,$$\sqrt{t}\psi$ and $\sqrt{t}\psi_x$ converge uniformly in compact subsets of $(0, \infty)$, and the limits are given by \begin{equation*}
        \begin{aligned}
            &\displaystyle{\lim_{t \to \infty}}\sqrt{t}\psi(x,t)= \frac{\phi_0(\infty)}{\sqrt{\pi}}x-\frac{2}{\alpha\sqrt{\pi}}\Big[\phi_0(-\infty)+\phi_0(\infty)\Big],\\&
            \displaystyle{\lim_{t \to \infty}}\sqrt{t}\psi_x(x,t)=\frac{\phi_0(\infty)}{\sqrt{\pi}}.
        \end{aligned}
    \end{equation*}
\end{lemma}

Next, we characterize the asymptotic behavior of the weak solution of \eqref{IVP}.
\subsubsection{Proof of the theorem \eqref{limit of solution}}
\begin{proof}
   According to theorem \eqref{explicit solution}, the unique weak solution is given by \begin{equation*}
         u(x,t)=\begin{cases}
           -2\frac{\phi_x(x,t)}{\phi(x,t)}, &x<0,t>0\\
            -2\frac{\psi_x(x,t)}{\psi(x,t)}-\alpha, &x>0,t>0.
        \end{cases}
    \end{equation*} Therefore, for $x<0,t>0$, lemma \eqref{approximation of phi} gives \begin{equation*}
        \begin{aligned}
            \displaystyle{\lim_{t \to \infty}}u(x,t)&=-2\displaystyle{\lim_{t \to \infty}}\frac{\sqrt{t}\phi_x(x,t)}{\sqrt{t}\phi(x,t)}\\
            &=\frac{2\phi_0(-\infty)}{-\phi_0(-\infty)x-\frac{2}{\alpha}\Big[\phi_0(-\infty)+\phi_0(\infty)\Big]}\\
            &= \frac{-2\alpha\phi_0(-\infty)}{2\Big[\phi_0(-\infty)+\phi_0(\infty)\Big]+\alpha \phi_0(-\infty)x}
        \end{aligned}
    \end{equation*}  and for $x>0,t>0$, lemma \eqref{approximation of psi} gives \begin{equation*}
        \begin{aligned}
              \displaystyle{\lim_{t \to \infty}}u(x,t)&=-2 \displaystyle{\lim_{t \to \infty}}\frac{\sqrt{t}\psi_x(x,t)}{\sqrt{t}\psi(x,t)}-\alpha\\
              & =\frac{-2\alpha \phi_0(\infty)}{\alpha \phi_0(\infty)x-2\Big[\phi_0(-\infty)+\phi_0(\infty)\Big]}-\alpha\\
              &=\frac{\alpha\Big[2 \phi_0(-\infty)-\alpha \phi_0(\infty)x\Big]}{\alpha\phi_0(\infty)x-2\Big[\phi_0(-\infty)+\phi_0(\infty)\Big]}.
        \end{aligned}
    \end{equation*}
\end{proof}
\subsection{Graphical representation of the solution}
In this section, we show the graph which depicts that the explicit solution converges to the steady state solution for a large value of $t$ for $u_0=\begin{cases}
    0, & x<0\\
    -\alpha, & x>0
\end{cases}$ and $\alpha<0.$ 
\noindent

    For the above initial data we get 
  \begin{equation*}
        \begin{aligned}
            G(t)
            &=\frac{1}{\sqrt{\pi}}\int_0^t\frac{1}{\sqrt{s(t-s)}}\Big[\frac{1}{\sqrt{\pi}}+\frac{\alpha\sqrt{t-s}}{4}e^{\frac{\alpha^2(t-s)}{16}}\operatorname{erfc}\Big(\frac{-\alpha\sqrt{t-s}}{4}\Big)\Big] ds\\
            &=\frac{1}{\sqrt{\pi}}\int_0^t\frac{1}{\sqrt{t-s}}\Big[\frac{1}{\sqrt{\pi s}}+\frac{\alpha}{4}e^{\frac{\alpha^2s}{16}}\operatorname{erfc}\Big(\frac{-\alpha\sqrt{s}}{4}\Big)\Big] ds\\
            &=\frac{1}{\pi}\int_0^1\frac{1}{\sqrt{{v(1-v)}}}dv+\frac{\alpha}{4\sqrt{\pi}}\int_0^t\frac{h(s)}{\sqrt{t-s}}ds=h(t)
        \end{aligned}
    \end{equation*}
where the last line follows from the calculation \eqref{calculation}.

\noindent
Hence, from \eqref{derivative of h}  we get $G^{\prime}(t)=\frac{\alpha}{4\sqrt{\pi t}}+\frac{\alpha^2}{16}h(t).$ Therefore 
\begin{equation*}
    \begin{aligned}
         &\phi(x,t)=\frac{-1}{2\sqrt{\pi t}}\int_0^\infty \Big\{e^{-\frac{(x-\xi)^2}{4t}}-e^{-\frac{(x+\xi)^2}{4t}}\Big\}d\xi-\frac{x}{2\sqrt{\pi}}\int_0^te^{-\frac{x^2}{4 (t-\tau)}}\frac{h(\tau)}{(t-\tau)^{\frac{3}{2}}}d\tau\\
         & \psi(x,t)=\frac{1}{2\sqrt{\pi t}}\int_0^\infty \Big\{e^{-\frac{(x-\xi)^2}{4t}}-e^{-\frac{(x+\xi)^2}{4t}}\Big\}d\xi+\frac{x}{2\sqrt{\pi}}\int_0^te^{-\frac{x^2}{4(t-\tau)}}\frac{h(\tau)}{(t-\tau)^{\frac{3}{2}}}d\tau\\
         & \phi_x(x,t)  =\frac{-1}{4\sqrt{\pi}(t)^{\frac{3}{2}}}\int_0^{\infty}\Big\{(\xi-x)e^{-\frac{(x-\xi)^2}{4t}}+(x+\xi)e^{-\frac{(x+\xi)^2}{4t}}\Big\}d\xi\\&+\frac{\alpha}{4\pi}\int_0^t\frac{1}{\sqrt{\tau(t-\tau)}}e^{\frac{-x^2}{4(t-\tau)} }d\tau+\frac{\alpha^2}{16\sqrt{\pi}}\int_0^t\frac{1}{\sqrt{t-\tau}}e^{\frac{-x^2}{4(t-\tau)} }h(\tau)d\tau+\frac{1}{\sqrt{\pi t}}e^{\frac{-x^2}{4t}}\\
           &  \psi_x(x,t) =\frac{1}{4\sqrt{\pi}(t)^{\frac{3}{2}}}\int_0^{\infty}\Big\{(\xi-x)e^{-\frac{(x-\xi)^2}{4t}}+(x+\xi)e^{-\frac{(x+\xi)^2}{4t}}\Big\}d\xi\\&-\frac{\alpha}{4\pi}\int_0^t\frac{1}{\sqrt{\tau(t-\tau)}}e^{\frac{-x^2}{4(t-\tau)} }d\tau-\frac{\alpha^2}{16\sqrt{\pi}}\int_0^t\frac{1}{\sqrt{t-\tau}}e^{\frac{-x^2}{4(t-\tau)} }h(\tau)d\tau-\frac{1}{\sqrt{\pi t}}e^{\frac{-x^2}{4t}}.
    \end{aligned}
\end{equation*}
Hence, from \eqref{explicit formula} we get the explicit formula. In the graph below solid lines are the explicit solution for $\alpha=-1$ and the dashed lines are the steady state solution 
\begin{equation*}
      \displaystyle{\lim_{t \to \infty}}  u(x,t)=
      \begin{cases}
          \frac{2 }{4-x}, &x<0\\
         \frac{x+2}{x+4}, &x>0.
      \end{cases}
\end{equation*}

\begin{figure}[htbp]
    \centering
    
    \begin{subfigure}[b]{0.32\linewidth}
        \centering
        \includegraphics[width=\textwidth]{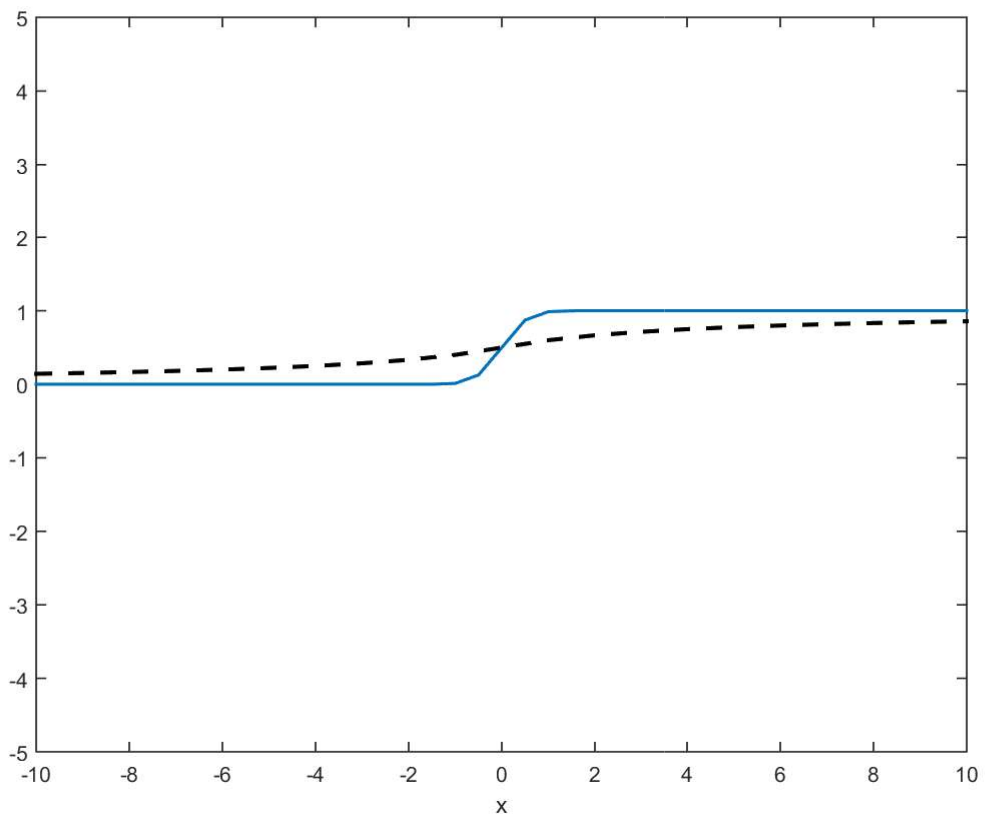}
        \caption{$t = 0.1$}
    \end{subfigure}
    \hfill
   \begin{subfigure}[b]{0.32\linewidth}
        \centering
        \includegraphics[width=\textwidth]{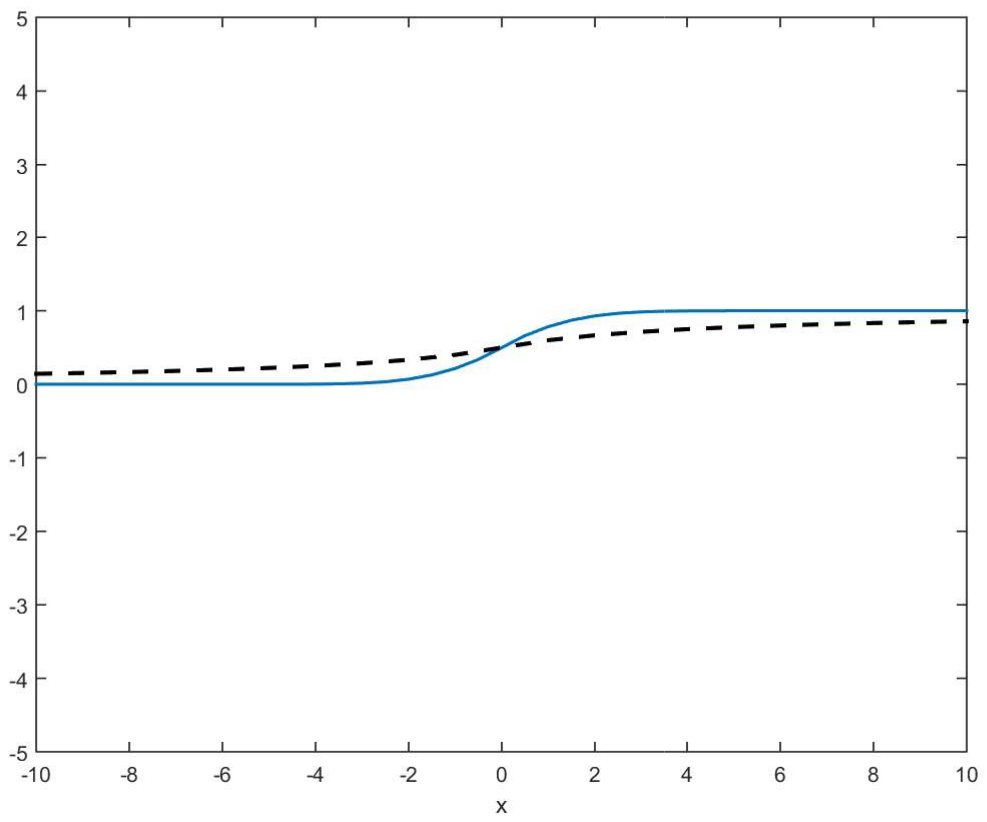}
        \caption{$t = 1$}
    \end{subfigure}
    \begin{subfigure}[b]{0.32\linewidth}
        \centering
        \includegraphics[width=\textwidth]{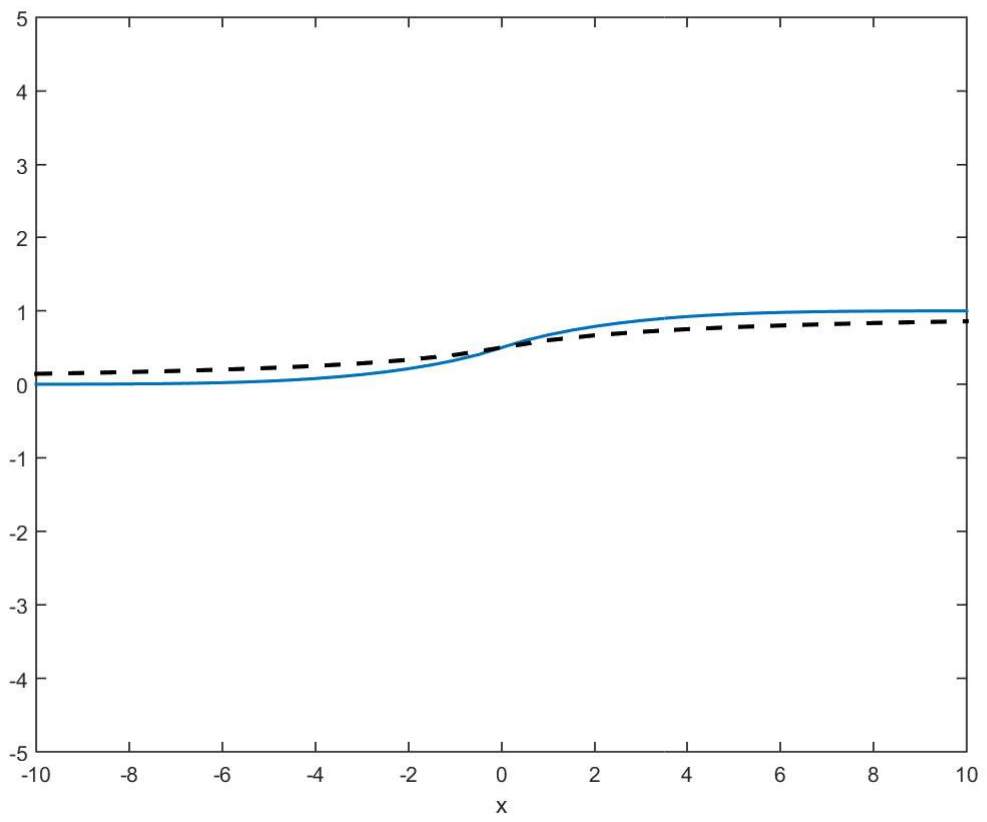}
        \caption{$t = 5$}
    \end{subfigure}
    \begin{subfigure}[b]{0.32\linewidth}
        \centering
        \includegraphics[width=\textwidth]{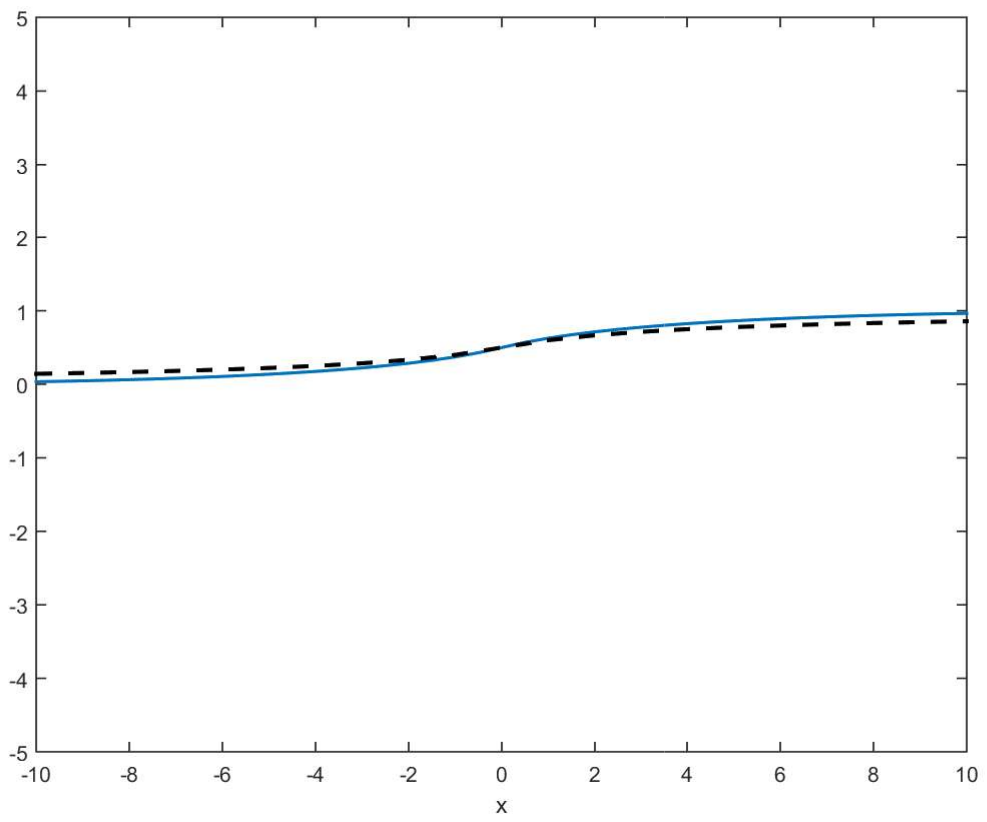}
        \caption{$t = 20$}
    \end{subfigure}
     \hfill 
    \begin{subfigure}[b]{0.32\linewidth}
        \centering
        \includegraphics[width=\textwidth]{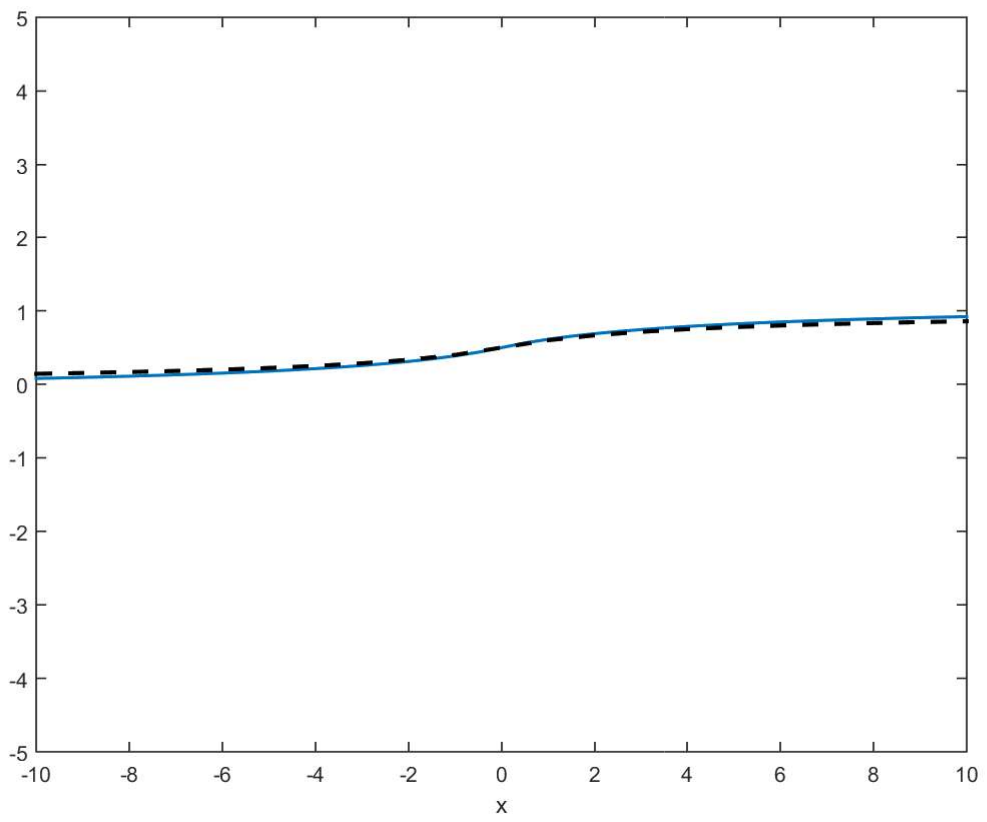}
        \caption{$t = 50$}
    \end{subfigure}
    \hfill
   \begin{subfigure}[b]{0.32\linewidth}
        \centering
        \includegraphics[width=\textwidth]{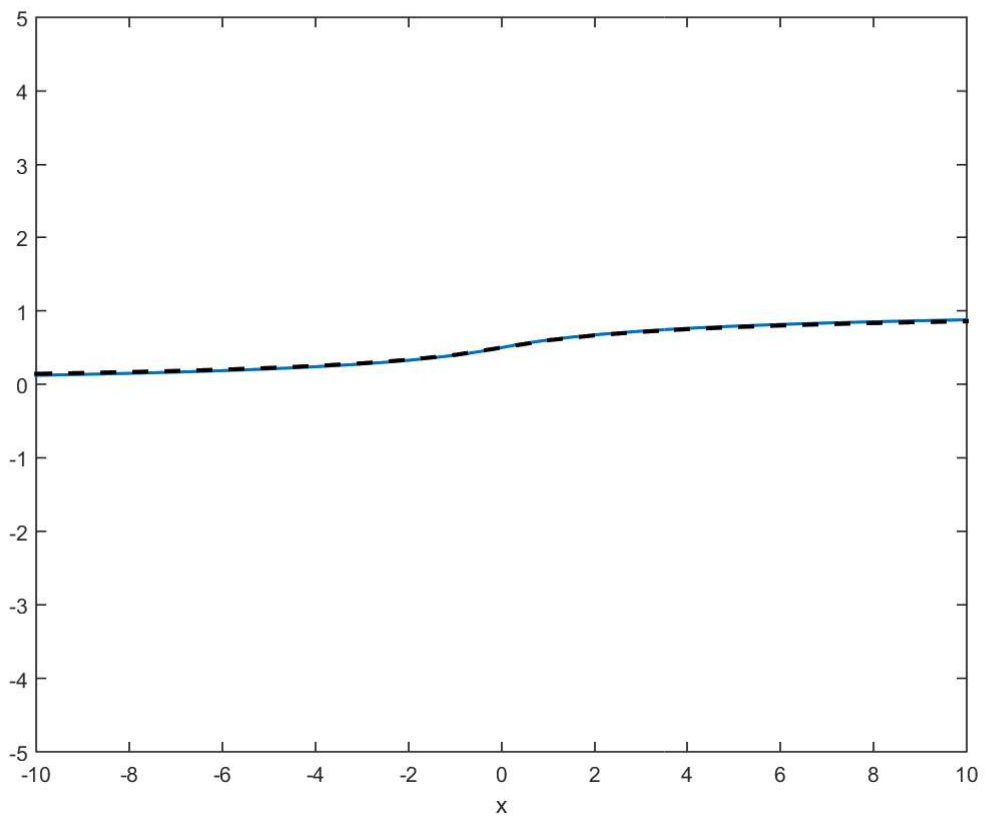}
        \caption{$t = 200$}
    \end{subfigure}

    \caption{Solid lines are the explicit solution at a given time $t$ and for the initial data $u_0=\begin{cases}
    0, & x<0\\
    -\alpha, & x>0
\end{cases}, \alpha=-1$ and dashed lines are the steady state as in the last theorem. The solution approaches to the steady state.}
    \label{fig:1}
\end{figure}

\newpage
  \section{Large time behavior of the solution for $\alpha>0$}
\subsection{Asymptotic behavior of the intermediate function $G$}
We assume conditions \eqref{extracondition} along with $\alpha>0.$ Then the boundary condition $G(t)$ satisfies an asymptotic behavior, which is given by the following lemma.
\begin{lemma}\label{approximation of G-1}
      Under assumption \eqref{extracondition} along with $\alpha>0,$  $G(t)$ obtained in lemma \eqref{properties of G}  satisfies the following: $$\displaystyle{\lim_{t \to \infty}}\sqrt{t}e^{-\frac{\alpha^2t}{16}}G(t)=\frac{\alpha}{2\sqrt{\pi}}\int_0^{\infty}\int_0^{\infty}e^{-\frac{\alpha^2s}{16}}ue^{-u^2}\theta(2u\sqrt{s})duds.$$
\end{lemma}
\begin{proof}
By a similar calculation as in the previous section, we get 
\begin{equation*}
       \displaystyle{\lim_{t\to \infty}}\sqrt{t}e^{-\frac{\alpha^2t}{16}}G(t)=  \displaystyle{\lim_{t\to \infty}}B_1
\end{equation*}
where  
\begin{equation*}
        \begin{aligned}
            B_1&=\frac{e^{-\frac{\alpha^2t}{16}}}{\sqrt{\pi}}\int_0^{t}\frac{H\Big(\frac{-\alpha\sqrt{s}}{4}\Big)}{\sqrt{s}}\sqrt{t-s}f_1(t-s)ds\\
          &=\frac{e^{-\frac{\alpha^2t}{16}}}{\sqrt{\pi}}\int_0^{t}\frac{1}{\sqrt{s}}\Big[\frac{1}{\sqrt{\pi}}+\frac{\alpha\sqrt{s}}{4}e^{\frac{\alpha^2s}{16}}\operatorname{erfc}\Big(\frac{-\alpha\sqrt{s}}{4}\Big)\Big]\sqrt{t-s}f_1(t-s)ds\\
          &=\frac{e^{-\frac{\alpha^2t}{16}}}{\pi}\int_0^{t}\frac{1}{\sqrt{s}}\sqrt{t-s}f_1(t-s)ds+\frac{\alpha}{4\sqrt{\pi}}\int_0^{t}e^{-\frac{\alpha^2(t-s)}{16}}\operatorname{erfc}\Big(\frac{-\alpha\sqrt{s}}{4}\Big)\sqrt{t-s}f_1(t-s)ds\\
          &=K_1+K_2.
        \end{aligned}
    \end{equation*}
    Now $$\sqrt{t}f_1(t)=\int_0^{\infty}ue^{-u^2}\Big(\phi_0(-2u\sqrt{t})+\phi_0(2u\sqrt{t})\Big)du.$$

    \noindent
    Under the condition \eqref{extracondition} we get $$\displaystyle{\lim_{t\to \infty}}\sqrt{t}f_1(t)=\frac{1}{2}\Big[\phi_0(-\infty)+\phi_0(\infty)\Big].$$
    Hence $\displaystyle{\lim_{t\to \infty}}K_1=0.$ Next consider $K_2.$
    \begin{equation*}
        \begin{aligned}
            K_2&=\frac{\alpha}{4\sqrt{\pi}}\int_0^{t}e^{-\frac{\alpha^2(t-s)}{16}}\operatorname{erfc}\Big(\frac{-\alpha\sqrt{s}}{4}\Big)\sqrt{t-s}f_1(t-s)ds\\
            &=\frac{\alpha}{4\sqrt{\pi}}\int_0^{\infty}\chi_{[0,t]}e^{-\frac{\alpha^2s}{16}}\operatorname{erfc}\Big(\frac{-\alpha\sqrt{t-s}}{4}\Big)\sqrt{s}f_1(s)ds.
        \end{aligned}
    \end{equation*}
    Since integrand is bounded so we get 
    \begin{equation*}
        \begin{aligned}
            \displaystyle{\lim_{t\to \infty}}K_2=\frac{\alpha}{2\sqrt{\pi}}\int_0^{\infty}\int_0^{\infty}e^{-\frac{\alpha^2s}{16}}ue^{-u^2}\theta(2u\sqrt{s})duds.
        \end{aligned}
    \end{equation*}
    Therefore $\displaystyle{\lim_{t \to \infty}}\sqrt{t}e^{-\frac{\alpha^2t}{16}}G(t)=\frac{\alpha}{2\sqrt{\pi}}\int_0^{\infty}\int_0^{\infty}e^{-\frac{\alpha^2s}{16}}ue^{-u^2}\theta(2u\sqrt{s})duds.$

\end{proof}

\subsection{Asymptotic behavior of the weak solution}
In this section, we derive the asymptotic behavior of this solution using the asymptotic limit of $\sqrt{t}e^{-\frac{\alpha^2t}{16}}G(t)$. In order to obtain the asymptotic behavior of the solution, we first obtain the asymptotic behavior of $\phi,\psi$ and $\phi_x,\psi_x$ separately as in the following two lemmas.
\begin{lemma}\label{approximation of phi 1}
  Under assumption \eqref{extracondition} along with $\alpha>0,$ $\sqrt{t}e^{-\frac{\alpha^2t}{16}}\phi$ and $\sqrt{t}e^{-\frac{\alpha^2t}{16}}\phi_x$ converge uniformly on compact subsets of $(-\infty, 0)$,  and the limits are given by
  \begin{equation*}
        \begin{aligned}
            &\displaystyle{\lim_{t \to \infty}}\sqrt{t}e^{-\frac{\alpha^2t}{16}}\phi(x,t)=Le^{\frac{\alpha x}{4}},\\&
            \displaystyle{\lim_{t \to \infty}}\sqrt{t}e^{-\frac{\alpha^2t}{16}}\phi_x(x,t)=\frac{L\alpha}{4}e^{\frac{\alpha x}{4}}
        \end{aligned}
    \end{equation*} 
  where  $L:=\displaystyle{\lim_{t \to \infty}}\sqrt{t}e^{-\frac{\alpha^2t}{16}}G(t).$
\end{lemma}
\begin{proof}
 By a similar calculation as in lemma \eqref{approximation of phi} we get 
    \begin{equation*}
        \begin{aligned}
       \displaystyle{\lim_{t \to \infty}}  \sqrt{t}e^{-\frac{\alpha^2t}{16}}\phi(x,t)&= \displaystyle{\lim_{t \to \infty}}\Big[-\frac{xe^{-\frac{\alpha^2t}{16}}}{2\sqrt{\pi}}\int_0^t\frac{e^{-\frac{x^2}{4 \tau}}}{\tau^{\frac{3}{2}}}\sqrt{t-\tau}G(t-\tau)d\tau\Big]\\
        &= \displaystyle{\lim_{t \to \infty}}\Big[-\frac{x}{2\sqrt{\pi}}\int_0^{\infty}\chi_{[0,t]}\frac{e^{-\frac{x^2}{4 \tau}-\frac{\alpha^2\tau}{16}}}{\tau^{\frac{3}{2}}}\sqrt{t-\tau}e^{-\frac{\alpha^2(t-\tau)}{16}}G(t-\tau)d\tau\Big].
        \end{aligned}
    \end{equation*}
    Since $\frac{e^{-\frac{x^2}{4 \tau}-\frac{\alpha^2\tau}{16}}}{\tau^{\frac{3}{2}}}$ is integrable, by the dominated convergence theorem we get \begin{equation*}
        \begin{aligned}
            \displaystyle{\lim_{t \to \infty}}  \sqrt{t}e^{-\frac{\alpha^2t}{16}}\phi(x,t)&=-\frac{xL}{2\sqrt{\pi}}\int_0^{\infty}\frac{e^{-\frac{x^2}{4 \tau}-\frac{\alpha^2\tau}{16}}}{\tau^{\frac{3}{2}}}d\tau\\
            &=Le^{\frac{\alpha x}{4}}.
        \end{aligned}
    \end{equation*}
    Similarly,
    \begin{equation*}
        \begin{aligned}
              \displaystyle{\lim_{t \to \infty}}  \sqrt{t}e^{-\frac{\alpha^2t}{16}}\phi_x(x,t)&=\displaystyle{\lim_{t \to \infty}}\Big[\frac{x^2e^{-\frac{\alpha^2t}{16}}}{4\sqrt{\pi}}\int_0^{\infty}\frac{e^{-\frac{x^2}{4\tau}}}{\tau^{\frac{5}{2}}}\chi_{[0,t]}\sqrt{t-\tau}G(t-\tau)d\tau\\&-\frac{e^{-\frac{\alpha^2t}{16}}}{2\sqrt{\pi}}\int_0^{\infty}\frac{e^{-\frac{x^2}{4\tau}}}{\tau^{\frac{3}{2}}}\chi_{[0,t]}\sqrt{t-\tau}G(t-\tau)d\tau\Big]\\
              &=\displaystyle{\lim_{t \to \infty}}\Big[\frac{x^2}{4\sqrt{\pi}}\int_0^{\infty}\frac{e^{-\frac{x^2}{4\tau}-\frac{\alpha^2\tau}{16}}}{\tau^{\frac{5}{2}}}\chi_{[0,t]}\sqrt{t-\tau}e^{-\frac{\alpha^2(t-\tau)}{16}}G(t-\tau)d\tau\\&-\frac{1}{2\sqrt{\pi}}\int_0^{\infty}\frac{e^{-\frac{x^2}{4\tau}-\frac{\alpha^2\tau}{16}}}{\tau^{\frac{3}{2}}}\chi_{[0,t]}\sqrt{t-\tau}e^{-\frac{\alpha^2(t-\tau)}{16}}e^{-\frac{\alpha^2(t-\tau)}{16}}G(t-\tau)d\tau\Big]\\
              &=\frac{L\alpha}{4}e^{\frac{\alpha x}{4}}.
        \end{aligned}
    \end{equation*}
     \end{proof}
Similarly, we can prove the following lemma.
\begin{lemma}\label{approximation of psi 1}
    Under assumption \eqref{extracondition} along with $\alpha>0,$ $\sqrt{t}e^{-\frac{\alpha^2t}{16}}\psi$ and $\sqrt{t}e^{-\frac{\alpha^2t}{16}}\psi_x$ converge uniformly in compact subsets of $(0, \infty)$, and the limits are given by \begin{equation*}
        \begin{aligned}
            &\displaystyle{\lim_{t \to \infty}}\sqrt{t}e^{-\frac{\alpha^2t}{16}}\psi(x,t)=Le^{\frac{-\alpha x}{4}} ,\\&
            \displaystyle{\lim_{t \to \infty}}\sqrt{t}e^{-\frac{\alpha^2t}{16}}\psi_x(x,t)=-\frac{L\alpha}{4}e^{-\frac{\alpha x}{4}}.
        \end{aligned}
    \end{equation*}
\end{lemma}

Next, we characterize the asymptotic behavior of the weak solution of \eqref{IVP}.
\subsubsection{Proof of the theorem \eqref{limit of solution 1}}
\begin{proof}
   According to theorem \eqref{explicit solution}, the unique weak solution is given by \begin{equation*}
         u(x,t)=\begin{cases}
           -2\frac{\phi_x(x,t)}{\phi(x,t)}, &x<0,t>0\\
            -2\frac{\psi_x(x,t)}{\psi(x,t)}-\alpha, &x>0,t>0.
        \end{cases}
    \end{equation*} Therefore, for $x<0,t>0$, lemma \eqref{approximation of phi 1} gives 
    \begin{equation*}
            \displaystyle{\lim_{t \to \infty}}u(x,t)=    -2 \displaystyle{\lim_{t \to \infty}}\frac{\sqrt{t}e^{-\frac{\alpha^2t}{16}}\phi_x(x,t)}{\sqrt{t}e^{-\frac{\alpha^2t}{16}}\phi(x,t)}=-\frac{\alpha}{2}.
    \end{equation*}  and for $x>0,t>0$, lemma \eqref{approximation of psi 1} gives \begin{equation*}
              \displaystyle{\lim_{t \to \infty}}u(x,t)=-2 \displaystyle{\lim_{t \to \infty}}\frac{\sqrt{t}e^{-\frac{\alpha^2t}{16}}\psi_x(x,t)}{\sqrt{t}e^{-\frac{\alpha^2t}{16}}\psi(x,t)}-\alpha
            =-\frac{\alpha}{2}.
    \end{equation*}
\end{proof}

\subsection{Graphical representation of the solution}
The graph of the explicit solution for $\alpha=1$ is shown below. The solid curves represent the explicit solution, while the dashed curves represent the corresponding steady-state solution
$ \displaystyle{\lim_{t \to \infty}}  u(x,t)=-\frac{1}{2}.$

\begin{figure}[htbp]
    \centering
    
    \begin{subfigure}[b]{0.29\linewidth}
        \centering
        \includegraphics[width=\textwidth]{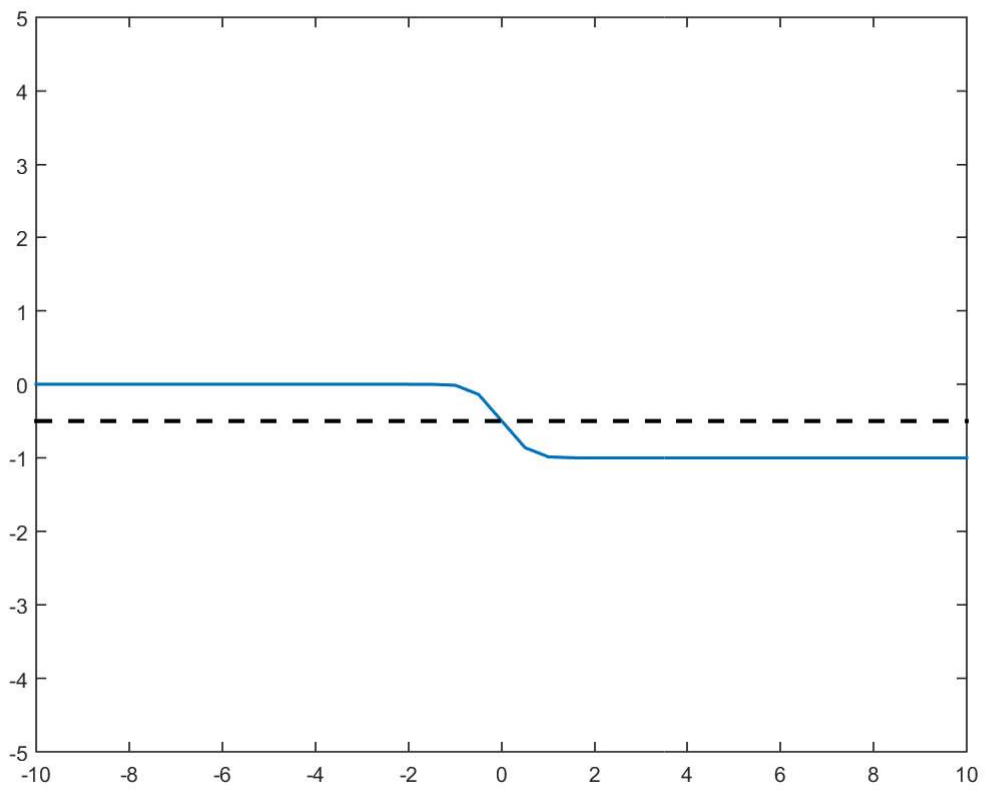}
        \caption{$t = 0.1$}
    \end{subfigure}
    \hfill
   \begin{subfigure}[b]{0.29\linewidth}
        \centering
        \includegraphics[width=\textwidth]{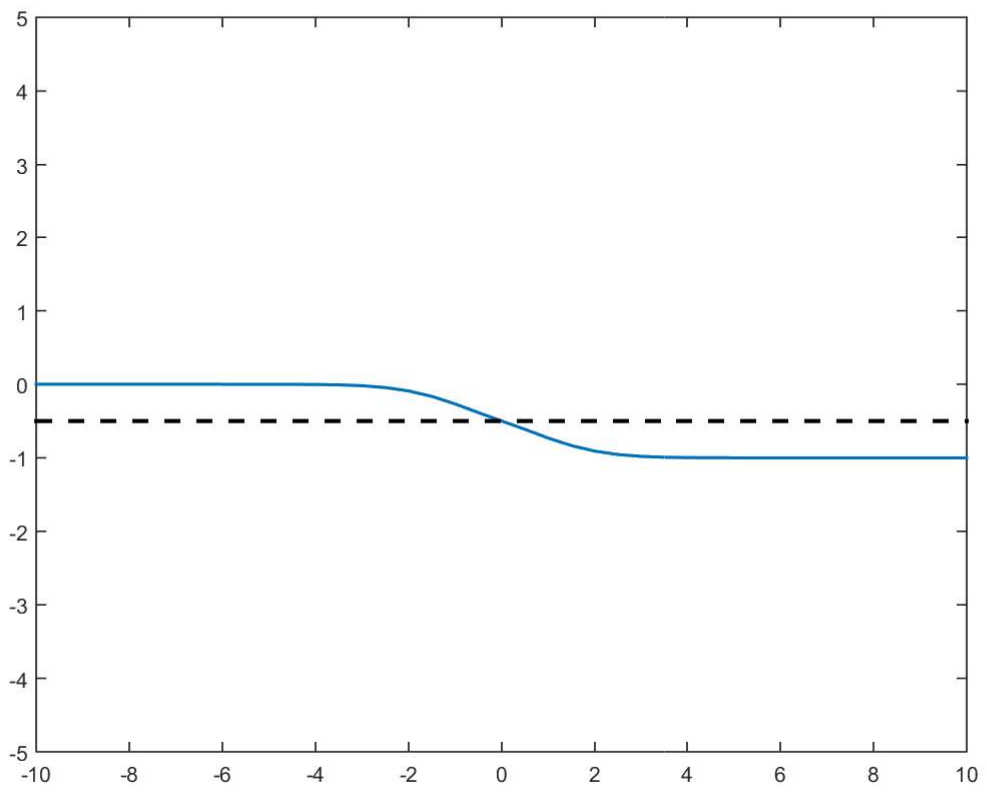}
        \caption{$t = 1$}
    \end{subfigure}
    \hfill
    \begin{subfigure}[b]{0.29\linewidth}
        \centering
        \includegraphics[width=\textwidth]{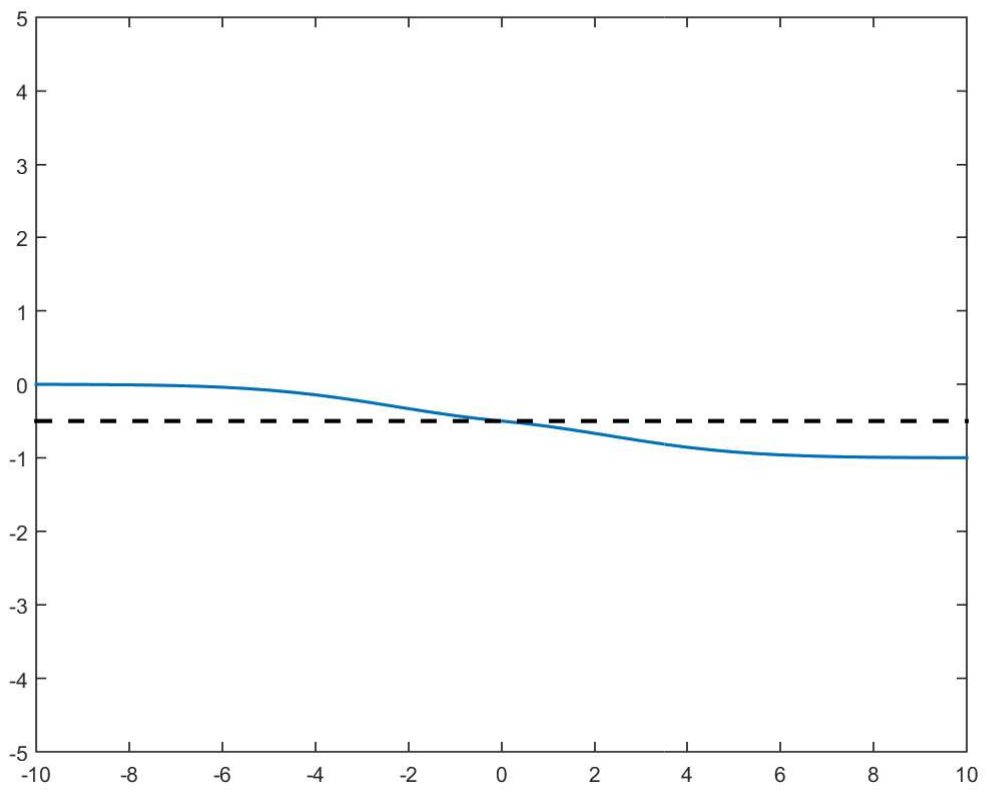}
        \caption{$t = 5$}
    \end{subfigure}
    \begin{subfigure}[b]{0.29\linewidth}
        \centering
        \includegraphics[width=\textwidth]{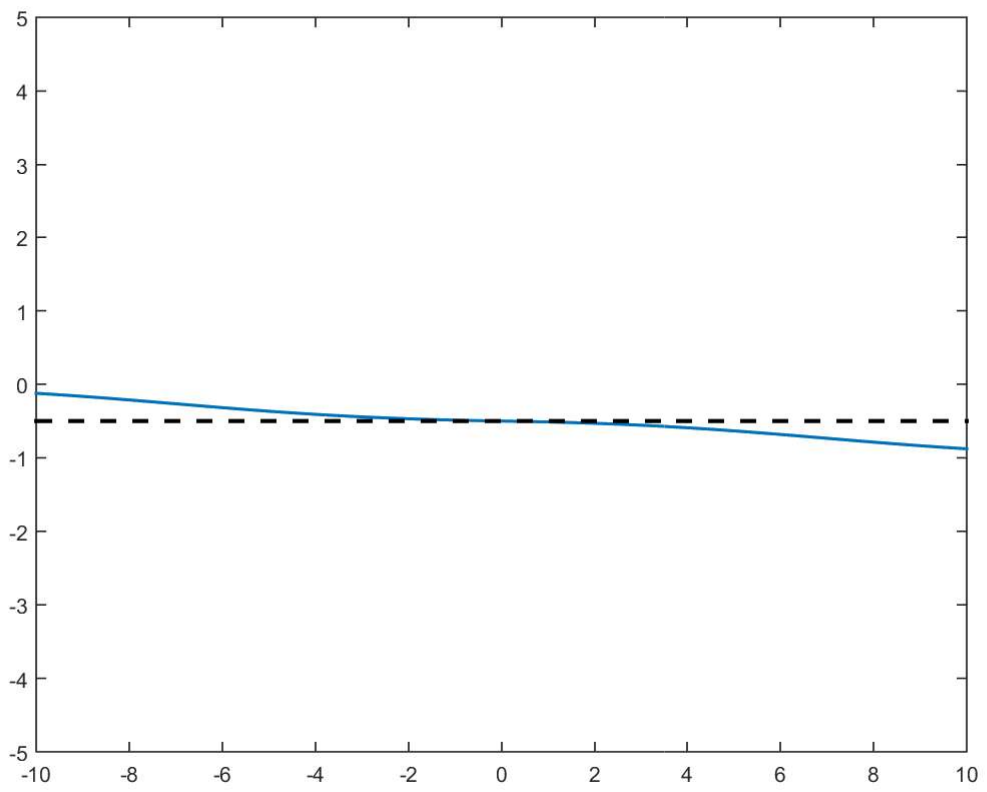}
        \caption{$t = 20$}
    \end{subfigure}
     \hfill 
    \begin{subfigure}[b]{0.29\linewidth}
        \centering
        \includegraphics[width=\textwidth]{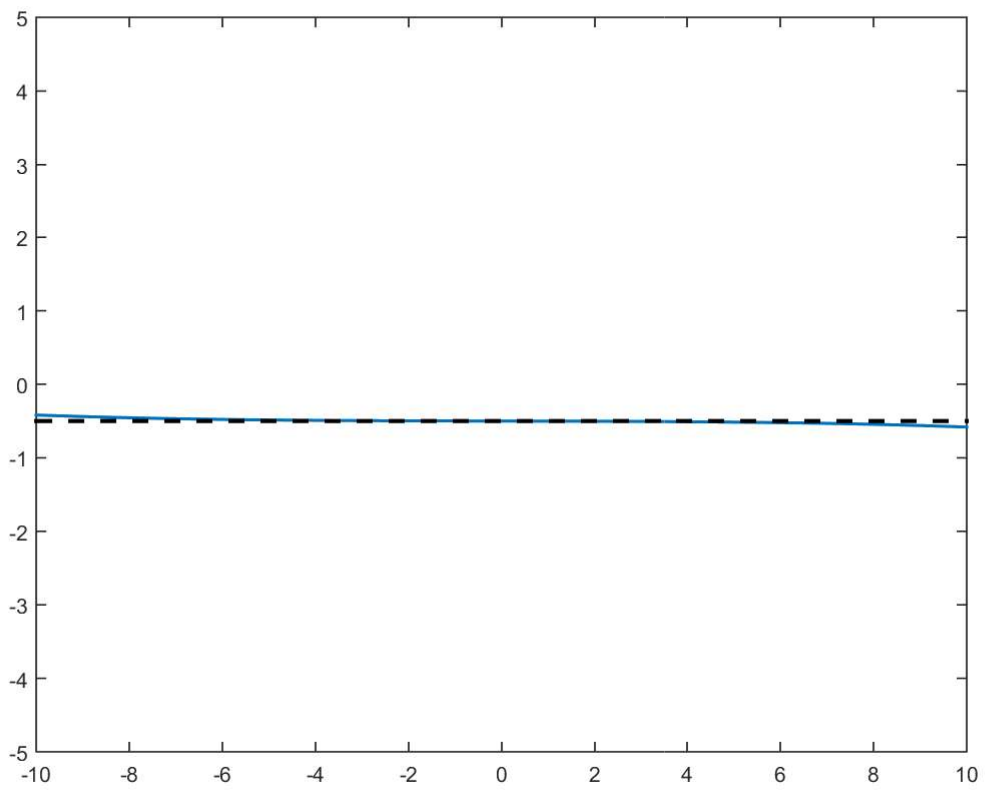}
        \caption{$t = 50$}
    \end{subfigure}
    \hfill
   \begin{subfigure}[b]{0.29\linewidth}
        \centering
        \includegraphics[width=\textwidth]{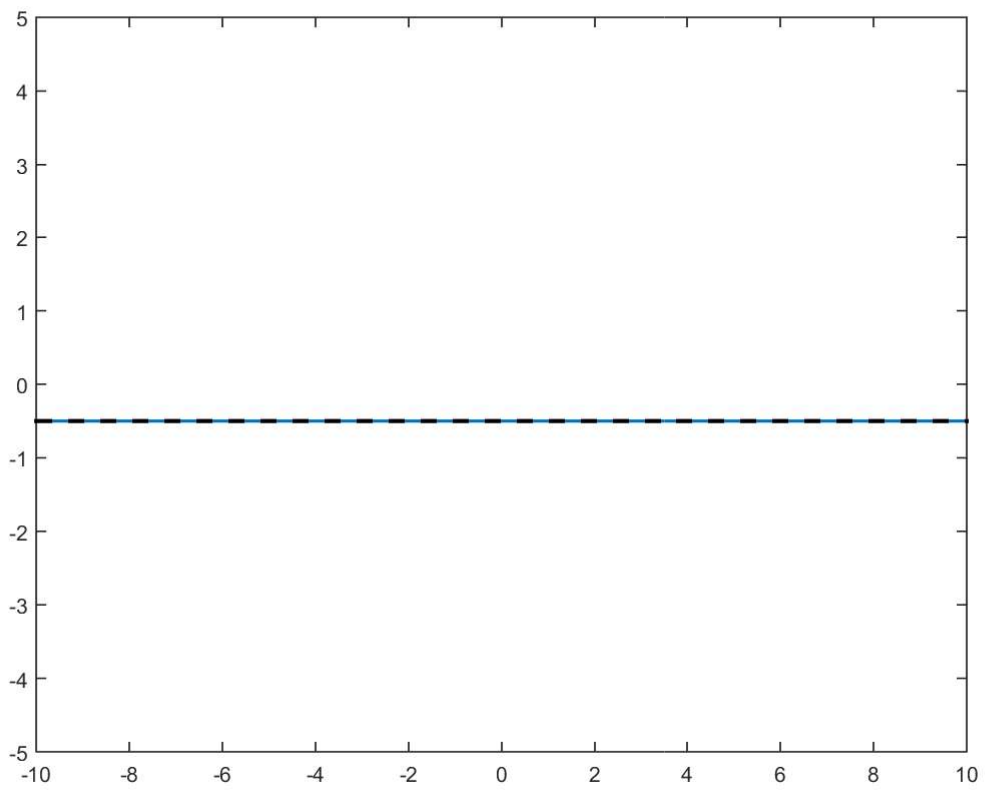}
        \caption{$t = 200$}
    \end{subfigure}
\end{figure}
\newpage
 \subsection*{Acknowledgments}
The first author acknowledges the financial support provided by the Council of Scientific and Industrial Research (CSIR), Government of India, under Junior Research Fellowship (JRF), File No. [09/1002(19997)/2024-EMR-I].

\bibliographystyle{plain}	
\bibliography{Reference}
\end{document}